\documentclass[12pt]{amsart}

\usepackage{amsmath,amssymb}
\usepackage{bm}
\usepackage{graphicx}
\usepackage{ascmac}
\usepackage{amsthm}
\usepackage{amsfonts}
\usepackage{mathrsfs}
\usepackage{mathtools}
\usepackage[all]{xy}
\mathtoolsset{showonlyrefs=true}

\theoremstyle{definition}
\newtheorem{thm}{Theorem}[section]
\newtheorem{dfn}[thm]{Definition}

\newtheorem{cor}[thm]{Corollary}
\newtheorem{prop}[thm]{Proposition}
\newtheorem{lem}[thm]{Lemma}
\newtheorem{rem}[thm]{Remark}
\newtheorem*{nota}{Notations}

\newtheorem*{claim}{Claim}
\newtheorem*{oftp}{Organization of this paper}
\newtheorem*{ack}{Acknowledgment}

\numberwithin{thm}{section}

\newcommand{\Zpn}{\mathbb{Z}_{>0}}
\newcommand{\Znn}{\mathbb{Z}_{\geq 0}}
\newcommand{\Z}{\mathbb{Z}}

\newcommand{\Zpt}{{\mathbb{Z}}_{p}^{\times}}

\newcommand{\Q}{\mathbb{Q}}

\newcommand{\Qp}{{\Q}_p}
\newcommand{\Qpt}{{\Q}_{p}^{\times}}

\newcommand{\Qpb}{\breve{\Q}_p}

\newcommand{\Ql}{{\mathbb{Q}}_{\ell}}

\newcommand{\R}{\mathbb{R}}
\newcommand{\C}{\mathbb{C}}
\newcommand{\F}{\mathbb{F}}

\newcommand{\A}{\mathbb{A}}

\newcommand{\G}{\mathbb{G}}
\newcommand{\bS}{\mathbb{S}}

\newcommand{\ebar}{\overline{e}}

\newcommand{\Fb}{\breve{F}}

\DeclareMathOperator{\Ker}{Ker}
\DeclareMathOperator{\Ima}{Im}
\DeclareMathOperator{\Hom}{Hom}

\DeclareMathOperator{\id}{id}

\DeclareMathOperator{\tr}{Tr}

\DeclareMathOperator{\ord}{ord}
\DeclareMathOperator{\Gal}{Gal}

\DeclareMathOperator{\GL}{GL}

\DeclareMathOperator{\N}{N}

\DeclareMathOperator{\sml}{sim}

\DeclareMathOperator{\Res}{Res}

\DeclareMathOperator{\diag}{diag}
\DeclareMathOperator{\rk}{rank}

\DeclareMathOperator{\Sh}{Sh}
\DeclareMathOperator{\val}{val}

\DeclareMathOperator{\sep}{sep}
\DeclareMathOperator{\der}{der}
\DeclareMathOperator{\ur}{ur}

\DeclareMathOperator{\Ind}{Ind}

\DeclareMathOperator{\spl}{spl}

\DeclareMathOperator{\Ram}{Ram}

\DeclareMathOperator{\Ad}{Ad}
\DeclareMathOperator{\tor}{tor}
\DeclareMathOperator{\In}{In}

\usepackage{geometry}
\title[Connected components of CM unitary Shimura varieties]{On the connected components of Shimura varieties for CM unitary groups in odd variables}
\author[Y.Oki]{Yasuhiro Oki}
\address{Graduate School of Mathematical Sciences, 
the University of Tokyo, 3-8-1 Komaba, Meguro-ku, Tokyo 153-8914, Japan. }
\email{oki@ms.u-tokyo.ac.jp}

\begin{document}
\maketitle
\begin{abstract}
We study the prime-to-$p$ Hecke action on the projective limit of the sets of connected components of Shimura varieties with fixed parahoric or Bruhat--Tits level at $p$. In particular, we construct infinitely many Shimura varieties for CM unitary groups in odd variables for which the considering actions are not transitive. We prove this result by giving negative examples on the question of Bruhat--Colliot-Th{\'e}l{\`e}ne--Sansuc--Tits or its variant, which is related to the weak approximation on tori over $\mathbb{Q}$. 
\end{abstract}

\tableofcontents

\section{Introduction}\label{intr}

Shimura varieties are complex manifolds introduced by Deligne (\cite{Deligne1979}), which are important for number theory. We recall the definition of them. Let $(G,X)$ be a Shimura datum, that is, a pair of a reductive connected groups over $\Q$ and a hermitian symmetric space satisfying some conditions. Then, for a compact open subgroup $K$ of $G(\A_{f})$, where $\A_{f}$ is the finite ad{\`e}le ring of $\Q$, the Shimura variety $\Sh_{K}(G,X)$ for $(G,X)$ with level $K$ is defined as follows: 
\begin{equation*}
\Sh_{K}(G,X):=G(\Q)\backslash X\times G(\A_{f})/K. 
\end{equation*}
We denote by $\pi_{0}(\Sh_{K}(G,X))$ the set of connected components of $\Sh_{K}(G,X)$. 

Now take a parahoric or a Bruhat--Tits subgroup $K_p$ of $G(\Qp)$. Here a Bruhat--Tits subgroup means the \emph{full} stabilizer of a facet in the Bruhat--Tits building of $G\otimes_{\Q}\Qp$. Then there is a right action of $G(\A_{f}^{p})$ on the projective system $\{\Sh_{K^pK_p}(G,X)\}_{K^p}$, which is called as the prime-to-$p$ Hecke action. Here $\A_f^p$ is the finite ad{\`e}le ring of $\Q$ without $p$-component. Put
\begin{equation*}
\pi_{0}(\Sh_{K_p}(G,X)):=\varprojlim_{K^p}\pi_{0}(\Sh_{K^pK_p}(G,X)). 
\end{equation*}
Then the prime-to-$p$ Hecke action on $\{\Sh_{K^pK_p}(G,X)\}_{K^p}$ induces a right action of $G(\A_{f}^{p})$ on $\pi_{0}(\Sh_{K_p}(G,X))$. We consider the following question: 
\begin{itemize}
\item[(T)] Is the action of $G(\A_{f}^{p})$ on $\pi_{0}(\Sh_{K_p}(G,X))$ transitive?
\end{itemize}

It is known to be affirmative if $G_{\Qp}$ is unramified, that is, $G_{\Qp}$ is quasi-split and splits over an unramified extension of $\Qp$. This follows from the fact that the weak approximation on $G$ at $p$ holds. In particular, (T) is affirmative for all Siegel modular varieties with parahoric level at $p$. On the other hand, there is an affirmative example even if $G_{\Qp}$ is ramified. A typical example is that $(G,X)$ is of PEL type attached to a hermitian space in odd variables over an imaginary quadratic field which is ramified at $p>2$. This follows from \cite[p.2756]{He2020}. However, no negative example was constructed explicitly. 

Note that the positivity of (T) implies some nice properties on the mod $p$ reduction of the Shimura variety $\Sh_{K^pK_p}(G,X)$ for any $K^p$. For example, if $(G,X)$ is of Hodge type, we can derive some properties on Newton strata and Kottwitz--Rapoport strata, and can develop the theory of EKOR stratification. See \cite{He2017}, \cite[\S 9]{He2020} and \cite[\S 8]{Zhou2020b}. On the other hand, if $(G,X)$ is of PEL type A or C such that $\Sh_{K^pK_p}$ is neither a Hilbert modular variety nor a Siegel modular variety, the above transitivity is necessary for the known cases for the Hecke orbit conjecture. See \cite[\S 1]{Xiao2020}. 

In this paper, we consider the question (T) in the case that $(G,X)$ is of PEL type attached to a hermitian space over a CM field of odd dimension. This study includes the latter known result as mentioned above. 

\subsection{Main theorems}\label{mtsv}

Let $L$ be a CM field, that is, a totally imaginary quadratic extension of a totally real field. We denote by $L^{+}$ the maximum totally real subfield of $L$. Take a subset $S$ of $\Hom(L,\C)$ which maps bijectively to $\Hom(L^{+},\R)$ by the canonical restriction map $\Hom(L,\C)\rightarrow \Hom(L^{+},\R)$. For $\varphi \in S$, we denote by $\overline{\varphi}$ the composite of the complex conjugation and $\varphi$. Let $V$ be an $L/L^{+}$-hermitian space such that $V\otimes_{L,\varphi}\C$ has signature $(r_{\varphi},r_{\overline{\varphi}})$ for $\varphi \in S$. Here we assume that the constant $r_{\varphi}+r_{\overline{\varphi}}$ ($\varphi \in S$) is an \emph{odd} number. We define a reductive connected group $G_{V}$ over $\Q$ consisting of $L$-linear automorphisms of $V$ which respect the hermitian form up to a \emph{rational} scalar multiple. See Section \ref{unsm} for the precise definition. We denote by $X_{V}$ the $G_{V}(\R)$-conjugacy class of the morphism
\begin{equation*}
\bS \rightarrow G_{V,\R};z\mapsto (\diag(z^{(r_{\varphi})},\overline{z}^{(r_{\overline{\varphi}})}))_{\varphi \in S}. 
\end{equation*}
Then $(G_{V},X_{V})$ is a Shimura datum, and hence we can consider the Shimura varieties $\Sh_{K}(G_{V},X_{V})$ and the set $\pi_{0}(\Sh_{K_p}(G_{V},X_{V}))$ equipped with an action of $G_{V}(\A_{f}^{p})$ for any parahoric or Bruhat--Tits subgroup $K_p$ of $G_{V}(\Qp)$. We write (T$_{V,K_p}$) for the above question for $(G,X)=(G_{V},X_{V})$. 

\begin{thm}\label{mth1}(Theorem \ref{shc1})
\emph{Suppose that $L$ is an abelian extension of $\Q$. Then (T$_{V,K_p}$) is affirmative for any $V$ over $L$ and $K_p$ if one of the following hold: 
\begin{enumerate}
\item $L/L^{+}$ is split at all $v\mid p$,
\item the ramification index of $L^{+}/\Q$ at $p$ is an odd number,
\item $p>2$ and $[L:\Q]\not\in 32\Z$,
\item $p=2$ and $[L:\Q]\not\in 8\Z$. 
\end{enumerate}}
\end{thm}

On the other hand, there are infinitely many negative examples in (T$_{V,K_p}$) for any $p$. 

\begin{thm}\label{mth2}(Theorem \ref{shc2})
\emph{
\begin{enumerate}
\item Assume $p>2$. For $d\in 32\Z$, there is an infinite family $\{L_j\}_{j\in J}$ of CM fields of degree $d$ that are abelian over $\Q$ such that (T$_{V,K_p}$) is negative for any $V$ over $L_j$ and $K_p$. Moreover, if $p\equiv 1\bmod 4$, then both the sets
\begin{equation*}
J_{\ur}:=\{j\in J\mid L_j/L_j^{+}\text{ is unramified at all }v\mid p\}
\end{equation*}
and $J\setminus J_{\ur}$ are infinite. 
\item For $d\in 8\Z$, there is an infinite family $\{L_j\}_{j\in J'}$ of CM fields of degree $d$ that are abelian over $\Q$ such that (T$_{V,K_2}$) is negative for any $V$ over $L_j$ and $K_2$. Moreover, both the sets
\begin{equation*}
J'_{\ur}:=\{j\in J'\mid L_j/L_j^{+}\text{ is unramified at all }v\mid p\}
\end{equation*}
and $J'\setminus J'_{\ur}$ are infinite. 
\end{enumerate}}
\end{thm}

\begin{rem}
\begin{enumerate}
\item Theorem \ref{mth1} (ii) implies that (T$_{V,K_p}$) is affirmative for any $V$ over $L$ and $K_p$ in the following cases: 
\begin{itemize}
\item $L$ is an imaginary quadratic field,
\item $p\equiv -1\bmod 4$ and $L/L^{+}$ is ramified at all $v\mid p$. 
\end{itemize}
\item A CM field $L_j$ ($j\in J$) as in Theorem \ref{mth2} (i) is tamely ramified over $\Q$ at $p$ if $d\not\in p\Z$. Hence $G_{V}$ splits over a tamely ramified extension for all $V$ over $L_j$. 
\end{enumerate}
\end{rem}

\subsection{The question of Bruhat--Colliot-Th{\'e}l{\`e}ne--Sansuc--Tits}\label{mtgr}

Let $L_{0}$ be a global field, $v_0$ a place of $L_0$ and $T$ a torus over $L_0$. We denote by $K_{T,v_0}$ the maximum compact open subgroup of $T(L_{0,v_0})$, where $L_{0,v_0}$ is the completion of $L_0$ at $v_0$. Then the question of Bruhat--Colliot-Th{\'e}l{\`e}ne--Sansuc--Tits is as follows: 
\begin{itemize}
\item[(A)] Does we have $T(L_0)\cdot K_{T,v_0}=T(L_{0,v_0})$?
\end{itemize}
Note that this is formulated in the paper of Colliot-Th{\'e}l{\`e}ne--Sansuc as a question of Bruhat--Tits (\cite[Remark 8.3]{ColliotThelene1987}). 

On the other hand, let $K_{T,v_0}^{\circ}$ be the kernel of the Kottwitz map of $T\otimes_{L_{0}}L_{0,v_0}$ defined by \cite[7.1--7.3]{Kottwitz1997}, which is a subgroup of $K_{T,v_0}$ of finite index. We also consider a variant of (A) as follows: 
\begin{itemize}
\item[(A$^{\circ}$)\hspace{-4mm}] \hspace{4mm}Does we have $T(L_0)\cdot K_{T,v_0}^{\circ}=T(L_{0,v_0})$?
\end{itemize}

Let $L$ and $L^{+}$ be as in Section \ref{mtsv}. We define a $\Q$-torus
\begin{equation*}
T_{L/L^{+}}:=\{t\in \Res_{L/\Q}\G_m\mid \N_{L/L^{+}}(t)\in \G_m\}. 
\end{equation*}
Take a prime number $p$, and we denote by $K_{L/L^{+},p}$ the maximum compact open subgroup of $T_{L/L^{+}}(\Qp)$. Here we consider (A) and (A$^{\circ}$) for the $\Q$-torus $T=T_{L/L^{+}}$ and the prime $p$, which are denoted by (A$_{L/L^{+},p}$) and (A$_{L/L^{+},p}^{\circ}$) respectively.  Note that (A$_{L/L^{+},p}$) and (A$_{L/L^{+},p}^{\circ}$) are affirmative if the weak approximation on $T_{L/L^{+}}$ at $p$ holds. Then Theorems \ref{mth1} and \ref{mth2} are reduced to the assertions as follow respectively. 

\begin{thm}\label{mta1}(Theorem \ref{mtaa}, Theorem \ref{mtas}) 
\emph{Suppose that $L$ is an abelian extension of $\Q$. Then both (A$_{L/L^{+},p}$) and (A$_{L/L^{+},p}^{\circ}$) are affirmative if one of the following hold: 
\begin{enumerate}
\item $L/L^{+}$ is split at all $v\mid p$,
\item the ramification index of $L^{+}/\Q$ at $p$ is an odd number,
\item $p>2$ and $[L:\Q]\not\in 32\Z$,
\item $p=2$ and $[L:\Q]\not\in 8\Z$. 
\end{enumerate}}
\end{thm}

\begin{thm}\label{mta2}(Theorem \ref{ngmt}) 
\emph{
\begin{enumerate}
\item Assume $p>2$. For $d\in 32\Z$, there is an infinite family $\{L_j\}_{j\in J}$ of CM fields of degree $d$ that are abelian over $\Q$ such that both (A$_{L_j/L_j^{+},p}$) and (A$_{L/L^{+},p}^{\circ}$) are negative for any $j\in J$. Moreover, if $p\equiv 1\bmod 4$, then both the sets
\begin{equation*}
J_{\ur}:=\{j\in J\mid L_j/L_j^{+}\text{ is unramified at all }v\mid p\}
\end{equation*}
and $J\setminus J_{\ur}$ are infinite. 
\item For $d\in 8\Z$, there is an infinite family $\{L_j\}_{j\in J'}$ of CM fields of degree $d$ that are abelian over $\Q$ such that both (A$_{L_j/L_j^{+},2}$) and (A$_{L_j/L_j^{+},2}^{\circ}$) are negative for any $j\in J'$. Moreover, both the sets
\begin{equation*}
J'_{\ur}:=\{j\in J'\mid L_j/L_j^{+}\text{ is unramified at all }v\mid p\}
\end{equation*}
and $J'\setminus J'_{\ur}$ are infinite. 
\end{enumerate}}
\end{thm}

\begin{rem}
In the case that $L_0$ is the field of rational functions of one variable over a finite field containing a $4$-th root of unity, then a $4$-dimensional counterexample in (A) is given by Colliot-Th{\'e}l{\`e}ne and Suresh (\cite{ColliotThelene2007}). Theorem \ref{mta2} gives the same examples as above in the case $L_0=\Q$ and $v_0=p$. 
\end{rem}

The proof of Theorem \ref{mta1} is divided into two parts. The first part is the study on the purely local question proposed by \cite{ColliotThelene2007} and its variant, which will be denoted by ($R_{T}$) and ($R_{T}^{\circ}$) respectively for a torus $T$ over a non-archimedean local field. It turns out that the positivity of ($R_{T_{L/L^{+},\Qp}}$) and ($R_{T_{L/L^{+},\Qp}}^{\circ}$) imply those of (A$_{L/L^{+},p}$) and (A$_{L/L^{+},p}^{\circ}$) respectively. See Proposition \ref{r2wa}. The assertions (i), (ii) and (iv) can be proved by this part. However, in the case (iii), ($R_{T_{L/L^{+},\Qp}}$) and ($R_{T_{L/L^{+},\Qp}}^{\circ}$) may becomes negative if $[L:\Q]\in 8\Z$. The second part is a direct study of (A$_{L/L^{+},p}$) and (A$_{L/L^{+},p}^{\circ}$). If $L/L^{+}$ is unramified at all $v\mid p$, we prove that (A$_{L/L^{+},p}$) and (A$_{L/L^{+},p}^{\circ}$) are affirmative by studying the norm image of $L/L^{+}$. Here we use the global class field theory and the Chebotarev density theorem. In the ramified case, the proof of the positivity of (A$_{L/L^{+},p}^{\circ}$) is given by that of (A$_{L'/L'^{+},p}$) for another CM field $L'$ and an explicit study of the Kottwitz map of $T_{L/L^{+}}\otimes_{\Q}\Qp$. 

On the other hand, for a proof of Theorem \ref{mta2}, we give a sufficient condition for the negativity of (A$_{L/L^{+},p}$), and construct such $L$ for any $p$ by using the Chebotarev density theorem. Note that $L$ also gives the negativity of (A$_{L/L^{+},p}^{\circ}$) since $K_{L/L^{+},p}^{\circ}$ is contained in $K_{L/L^{+},p}$. 

\begin{rem}
It is natural to study (T$_{V,K_p}$), (A$_{L/L^{+},p}$) and (A$_{L/L^{+},p}^{\circ}$) for non-abelian $L$. Moreover, we can consider (T) for other Shimura varieties. However, we may not reduce it to the similar question as (A$_{L/L^{+},p}$) and (A$_{L/L^{+},p}^{\circ}$) by the same method as above. It occurs when $(G,X)$ is of PEL type $\mathrm{D}$. These are two of our future problems. 
\end{rem}

\begin{oftp}
In Section \ref{pddf}, we recall general notions on tori, and introduce the coflasque resolution and the $R$-equivalence. Section \ref{reqp} is the first technical heart of this paper. Here we introduce the questions ($R_{T}$) and ($R_{T}^{\circ}$) for tori $T$ over non-archimedean local fields, and consider them for $T=T_{L/L^{+}}\otimes_{\Q}\Qp$ where $L/\Q$ is abelian. The goal of this section is specifying all $L/L^{+}$ for which ($R_{T_{L/L^{+},\Qp}}$) and ($R_{T_{L/L^{+},\Qp}}^{\circ}$) are negative. The second technical heart of this paper is Section \ref{warf}. Here we prove Theorems \ref{mta1} and \ref{mta2} by using the results in Section \ref{reqp} and a global method. Finally, we derive Theorems \ref{mth1} and \ref{mth2} from Theorems \ref{mta1} and \ref{mta2} respectively in Section \ref{ccmt}. 
\end{oftp}

\begin{ack}
I would like to thank my advisor Yoichi Mieda for his constant support and encouragement.

This work was carried out with the support from the Program for Leading Graduate Schools, MEXT, Japan. This work was also supported by the JSPS Research Fellowship for Young Scientists and KAKENHI Grant Number 19J21728.
\end{ack}

\begin{nota}
\begin{itemize}
\item For a field $k_0$, we fix a separable closure $k_0^{\sep}$ of $k_0$. For a subextension $k/k_0$ of $k_{0}^{\sep}/k_0$ and a Galois extension $k'$ of $k$, we denote by $\Gal(k'/k)$ the Galois group of $k'/k$. Moreover, we set $\Gamma_{k}:=\Gal(k_0^{\sep}/k)$. 
\item For a complete discrete valuation field $F_0$, we denote by $\val_{F_0}$ the valuation map on $F_0$ so that $\val_{F_0}(\varpi)=1$, where $\varpi$ is a uniformizer of $F_0$. 
\item Let $F_0$ be a non-archimedean local field, that is, a complete discrete valuation field whose residue field is a finite field $\F_{q}$. We denote by $I_{F_0}$ the inertia subgroup of $\Gamma_{F_0}$. Moreover, if $F_0=\Qp$, we set $\ord_{p}:=\val_{\Qp}$. On the other hand, for a finite Galois extension $F/F_0$, we choose a lift $\sigma$ on $F$ of the $q$-th power map on $\F_{q}$. 
\item Let $M$ be an abelian group. We denote by $M_{\tor}$ the tosion part of $M$. Moreover, if $M$ is equipped with an action of a group $G$, then we write $M^{G}$ and $M_{G}$ for the $G$-invariant and the $G$-coinvariant parts of $M$ respectively. 
\item The symbol $\delta_{ij}$ is the Kronecker's delta, that is,
\begin{equation*}
\delta_{ij}=
\begin{cases}
1&\text{if }i=j,\\
0&\text{otherwise}. 
\end{cases}
\end{equation*}
\end{itemize}
\end{nota}

\section{Preliminaries}\label{pddf}

\subsection{General theory on tori}

Let $k_0$ be a field. Recall that a torus over $k_0$ is an algebraic group $T$ over $k_0$ satisfying $T\otimes_{k}k^{\sep}\cong \G_m^{n}$ for some $n\in \Zpn$. Here $\G_m$ is the multiplicative group scheme. For a torus $T$ over $k_0$, put
\begin{equation*}
X_{*}(T):=\Hom_{k^{\sep}}(\G_{m,k^{\sep}},T\otimes_{k}k^{\sep}). 
\end{equation*}
It is a finite free $\Z$-module. We equip $X_{*}(T)$ with a left action of $\Gamma_{k_0}$ by the Galois conjugation: 
\begin{equation*}
\Gamma_{k_0}\times X_{*}(T)\rightarrow X_{*}(T);(\tau,c)\mapsto \Ad(\id_{T}\otimes (\tau^{-1})^{*})(c)
\end{equation*}
Then the action factors through $\Gamma_{k}$ where $k$ is a finite separable extension of $k_0$. 

In the sequel, we equip $\Z$ with the trivial action on $\Gamma_{k_0}$. 

\begin{prop}\label{treq}
\emph{The assignment $T\mapsto X_{*}(T)$ gives an equivalence between the category of tori over $k_0$ and that of finite free $\Z$-modules with left actions of $\Gamma_{k_0}$. Its quasi-inverse is given by $M\mapsto (M\otimes_{\Z}k_{0}^{\sep,\times})^{\Gamma_{k_0}}$. Moreover, the following hold. 
\begin{enumerate}
\item The $k_0$-torus $T$ splits over a finite separable extension $k$ of $k_0$ if and only if $\Gamma_{k}$ acts trivially on $X_{*}(T)$. 
\item For a $k_0$-torus $T$, we denote by $T^{\spl}$ the maximal split subtorus of $T$. Then we have
\begin{equation*}
X_{*}(T^{\spl})=X_{*}(T)^{\Gamma_{k_0}}. 
\end{equation*}
\item For a finite separable extension $k$ of $k_0$ and a $k_0$-torus $T$, $X_{*}(T\otimes_{k_0}k)$ is the $\Z$-module $X_{*}(T)$ with the action of $\Gamma_{k}$ induced by that of $\Gamma_{k_0}$. 
\item For a finite separable extension $k$ of $k_0$ and a $k$-torus $T'$, we have
\begin{equation*}
X_{*}(\Res_{k/k_0}T')=\Ind^{\Gamma_{k_0}}_{\Gamma_{k}}X_{*}(T'). 
\end{equation*}
\end{enumerate}}
\end{prop}

We will also use the following: 

\begin{prop}\label{frrp}
\emph{Let $k/k_0$ be a finite separable extension, $T_0$ a $k_0$-torus and $T$ a $k$-torus. Then there is an isomorphism
\begin{equation*}
\Hom_{k_0}(\Res_{k/k_0}T,T_0)\xrightarrow{\cong} \Hom_{k}(T,T_0\otimes_{k_0}k),
\end{equation*}
which is functorial with respect to $T_0$ and $T$. }
\end{prop}

\begin{proof}
This follows from Proposition \ref{treq} and the Frobenius reciprocity law. 
\end{proof}

\begin{dfn}\label{tkkd}
Let $k^{+}$ be a finite {\'e}tale algebra over $k_0$, and $k$ a separable quadratic extension of $k^{+}$. Then we define $k$-tori
\begin{align*}
T_{k/k^{+},k_0}&:=\{t\in \Res_{k/k_0}\G_m\mid \N_{k/k^{+}}(t)\in \G_m\},\\
T_{k/k^{+},k_0}^{1}&:=\{t\in \Res_{k/k_0}\G_m\mid \N_{k/k^{+}}(t)=1\}. 
\end{align*}
\end{dfn}

\begin{prop}\label{trid}
\emph{
\begin{enumerate}
\item There is the following commutative diagram
\begin{equation*}
\xymatrix{
1\ar[r]& T_{k/k^{+},k_0}^{1}\ar[r]\ar@{=}[d]& T_{k/k^{+},k_0}\ar[r]^{\nu_{k/k^{+},k_0}} \ar[d]& \G_m \ar[r] \ar@{^{(}->}[d]& 1\\
1\ar[r]& T_{k/k^{+},k_0}^{1}\ar[r]& \Res_{k/k_0}\G_m \ar[r]^{\N_{k/k^{+}}\hspace{1.2mm}}& \Res_{k^{+}/k_{0}}\G_m \ar[r]& 1. }
\end{equation*}
Moreover, both the horizontal sequences are exact. 
\item If $k^{+}$ is a field and $k=k^{+}\times k^{+}$, then there is an isomorphism
\begin{equation*}
T_{k/k^{+},k_0}\cong \G_m \times \Res_{k^{+}/k_0}\G_m, 
\end{equation*}
and $\nu_{k/k^{+},k_0}$ is given by the first projection. In particular, $T_{k/k^{+},k_0}$ is induced over $k_0$. 
\item Let $k^{+}=k_1^{+}\times \cdots \times k_r^{+}$ and $k=k_1\times \cdots \times k_r$, where $k_i^{+}$ is a field and $k_i$ is an {\'e}tale quadratic algebra over $k_i^{+}$ for any $i$. Then the following diagram is Cartesian: 
\begin{equation*}
\xymatrix@C=46pt{
T_{k/k^{+},k_0} \ar[r]^{\nu_{k/k^{+},k_0}} \ar[d]^{\diag_0} & \G_m \ar[d]^{\diag}\\
\prod_{i=1}^{r}T_{k_i/k_i^{+},k_0} \ar[r]^{\hspace{4mm}(\nu_{k_i/k_i^{+},k_0})_{i}} & \prod_{n}\G_{m}. }
\end{equation*}
Here $\diag_0$ is the canonical injection and $\diag$ is the diagonal map. 
\end{enumerate}}
\end{prop}

\begin{proof}
(i): This follows by definition. 

(ii): By assumption, the diagram
\begin{equation*}
\xymatrix{
\Res_{k/k_0}\G_m\ar[rd]_{\N_{k/k^{+}}}\ar[r]^{\cong\hspace{5mm}} & (\Res_{k^{+}/k_0}\G_m)^{2} \ar[d]^{(t,t')\mapsto tt'} \\
& \Res_{k^{+}/k_0}\G_m }
\end{equation*}
is commutative. Hence the assertion follows from this diagram and (i). 

(iii): This follows by the definitions of $T_{k/k^{+},k_0}$ and $T_{k_i/k_i^{+},k_0}$. 
\end{proof}

\subsection{Coflasque resolutions of tori and $R$-equivalence}\label{cofl}

\begin{dfn}
Let $T$ be a torus over $k_0$. 
\begin{itemize}
\item We say that $T$ is \emph{induced} over $k_0$ if there is an isomorphism
\begin{equation*}
T\cong \Res_{k/k_0}\G_m
\end{equation*}
for some finite {\'e}tale algebra $k$ over $k_0$. 
\item We say that $T$ is \emph{coflasque} over $k_0$ if $H^{1}(k,X_{*}(T))=0$ for any finite extension $k$ of $k_0$ (here $H^{1}$ is the first Galois cohomology). 
\end{itemize}
\end{dfn}

\begin{dfn}
A \emph{coflasque resolution} of a torus $T$ over $k_0$ is an exact sequence of tori over $k_0$
\begin{equation*}
1\rightarrow F\rightarrow P\rightarrow T\rightarrow 1,
\end{equation*}
which remains exact after taking $X_{*}(\cdot)$ such that $F$ and $P$ are coflasque and induced over $k_0$ respectively. 
\end{dfn}
In the following, we give a typical example of coflasque resolution of $T$. Take a splitting field $k$ of $T$ which is finite Galois over $k_0$. For a subextension $k'/k_0$ of $k/k_0$, we denote by $T_{k'}^{\spl}$ the maximal split torus of the $k'$-torus $T_{k'}$. Now we set
\begin{equation*}
P(T):=\prod_{k_0\subset k'\subset k}\Res_{k'/k_0}T_{k'}^{\spl}. 
\end{equation*}
By definition, $P(T)$ is induced over $k_0$. Moreover, let $P(T)\rightarrow T$ be the direct sum of the morphisms $\Res_{k'/k}T_{k'}^{\spl}\rightarrow T$ induced by the inclusion $T_{k'}^{\spl}\hookrightarrow T_{k'}$ under Proposition \ref{frrp} for all $k'$, and $F(T)$ the kernel of $P(T)\rightarrow T$. Then $F(T)$ is coflasque over $k_0$ by construction, and hence the exact sequence
\begin{equation*}
1\rightarrow F(T)\rightarrow P(T)\rightarrow T\rightarrow 1
\end{equation*}
gives a coflasque resolution of $T$. 

A description of the above sequence by means of Galois modules is as follows. Let $G:=\Gal(k/k_0)$. For a subgroup $H$ of $G$, 
\begin{equation*}
X_{*}(P(T)):=\bigoplus_{H<G}X_{*}(T)^{H},
\end{equation*}
and let $\nu_{*}$ be the induced map. Then $F_{*}:=\Ker(\nu_{*})$ is coflasque, and hence

\begin{dfn}
We define a subgroup $RT(k_0)$ of $T(k_0)$ as the image of $P(T)(k_0)\rightarrow T(k_0)$. 
\end{dfn}
An element of $RT(k_0)$ is said to be ``$R$-equivalent to $1$''. It is known that $RT(k_0)$ is independent of the choice of the coflasque resolution of $T$. See \cite[\S 5, Th{\'e}or{\`e}me 2]{ColliotThelene1977}. 

The following will be used in Section \ref{reqp}. 

\begin{prop}\label{crtw}
\emph{We have $RT(k_0)=T(k_0)$ if one of the following hold: 
\begin{enumerate}
\item $T$ is induced over $k_0$,
\item $T$ splits over a cyclic extension of $k_0$. 
\end{enumerate}}
\end{prop}

\begin{proof}
(i): This follows from the definition of $RT(k_0)=T(k_0)$. 

(ii): This is essentially obtained by \cite{Endo1974}. See also \cite[\S 1, p.276]{ColliotThelene2007}. 
\end{proof}

In the following, we analyze the homomorphisms
\begin{equation*}
P(T_{k/k^{+},k_0})\rightarrow T_{k/k^{+},k_0},\quad P(T_{k/k^{+},k_0}^{1})\rightarrow T_{k/k^{+},k_0}^{1}
\end{equation*}
and their images on $k_0$-valued points, that is, $RT_{k/k^{+},k_0}(k_0)$ and $RT_{k/k^{+},k_0}^{1}(k_0)$. 

\begin{prop}\label{bmrt}
\emph{The subgroup $RT_{k/k^{+},k_0}(k_0)$ contains $k_0^{\times}$. }
\end{prop}

\begin{proof}
This follows from that the canonical homomorphism $\G_m\hookrightarrow T_{k/k^{+},k_0}$ is a direct summand of the homomorphism $P(T_{k/k^{+},k_0})\rightarrow T_{k/k^{+},k_0}$. 
\end{proof}

\begin{prop}\label{n1rt}
\emph{Let $k^{+}=\prod_{r}k_1^{+}$ and $k=\prod_{r}k_1$, where $k_1^{+}$ is a field and $k_1$ is an {\'e}tale quadratic algebra over $k_1^{+}$. 
\begin{enumerate}
\item The following diagram is Cartesian: 
\begin{equation*}
\xymatrix{
RT_{k/k^{+},k_0}(k_0)\ar[r]^{\diag_0\hspace{4mm}}\ar[d]_{\nu_{k/k^{+},F_0}}&\prod_{n}RT_{k_1/k_1^{+},k_0}(k_0)\ar[d]^{\prod_{n}\nu_{k_1/k_1^{+},k_0}}\\
k_0^{\times}\ar[r]^{\diag\hspace{4mm}}&\prod_{n}k_0^{\times}.}
\end{equation*}
Here $\diag$ and $\diag_{0}$ are homomorphisms as in Proposition \ref{trid} (iii). 
\item The composite
\begin{equation*}
P(T_{k/k^{+},k_0}^{1})\rightarrow T_{k/k^{+},k_0}^{1}\hookrightarrow T_{k/k^{+},k_0}
\end{equation*}
is a direct summand of $P(T_{k/k^{+},k_0})\rightarrow T_{k/k^{+},k_0}$. 
\item If both $k_1^{+}/k_0$ and $k_1/k_0$ are Galois, then we have
\begin{equation*}
RT_{k/k^{+},k_0}^{1}(k_0)=T_{k/k^{+},k_0}^{1}(k_0)\subset RT_{k/k^{+},k_0}(k_0). 
\end{equation*}
\end{enumerate}}
\end{prop}

\begin{proof}
(i): Take a splitting field $\widetilde{k}$ of $T_{k/k^{+},k_0}$ which is finite Galois over $k_0$. Then the diagram
\begin{equation*}
\xymatrix{
(T_{k/k^{+},k_0})_{k'}^{\spl}\ar[r]^{\diag_0\hspace{4mm}}\ar[d]_{\nu_{k/k^{+},k_0}}&\prod_{n}(T_{k_1/k_1^{+},k_0})_{k'}^{\spl}\ar[d]^{\prod_{n}\nu_{k_1/k_1^{+},k_0}}\\
\G_{m,k'}\ar[r]^{\diag\hspace{4mm}}&\prod_{n}\G_{m,k'}}
\end{equation*}
is Cartesian for any subextension $k'/k_0$ of $\widetilde{k}/k_0$. Hence we obtain the desired diagram. 

(ii): The assertion is a consequence of the definitions of $P(T_{k/k^{+},k_0}^{1})\rightarrow T_{k/k^{+},k_0}^{1}$ and $P(T_{k/k^{+},k_0})\rightarrow T_{k/k^{+},k_0}$. 

(iii): Write $\Gal(k/k_0)=\{\id,\tau\}$. By hypothesis, the homomorphism
\begin{equation*}
\psi_{i}\colon \Res_{k_1/k_0}\G_m\rightarrow T_{k/k^{+},k_0};x\mapsto (\delta_{ij}x\tau(x)^{-1})_{1\leq j\leq r}
\end{equation*}
is a direct summand of the homomorphisms $P(T_{k/k^{+},k_0}^{1})\rightarrow T_{k/k^{+},k_0}^{1}$ and $P(T_{k/k^{+},k_0})\rightarrow T_{k/k^{+},k_0}$ for any $i\in \{1,\ldots,r\}$. Hence the assertion follows from the Hilbert satz 90. 
\end{proof}

\section{$R$-equivalence on tori over non-archimedean local fields}\label{reqp}

Throughout this section, let $F_0$ be a non-archimedean local field, $F^{+}$ a finite {\'e}tale algebra over $F_0$ and $F$ an {\'e}tale quadratic algebra over $F^{+}$. 

\subsection{Abelian extensions of non-archimedean local fields}\label{abex}

Here we assume that both $F/F_0$ is an abelian extension of a power of $2$, and set $G:=\Gal(F/F_0)$. We denote by $I\subset G$ the inertia group of $F/F_0$, and write $F^{\ur}/F_0$ for the subextension of $F/F_0$ corresponding to $I$. Note that $F^{\ur}$ is the maximal unramified subextension of $F/F_0$. Write $[F^{\ur}:F_0]=2^{n}$ where $n\in \Znn$. 

\begin{prop}\label{abna}
\emph{Suppose that $I$ is cyclic of order $2^{m}$, where $m\in \Znn$. Then $F/F_0$ satisfies one of the following: 
\begin{enumerate}
\item $G\cong \Z/2^{m}\times \Z/2^{n}$ which induces $I\cong \langle(1,0)\rangle$, 
\item $G\cong \Z/2^{u}\times \Z/2^{m+n-u}$ which induces $I\cong \langle(1,2^{n-u})\rangle$, where $0\leq u<\min\{m,n\}$. 
\end{enumerate}}
\end{prop}

\begin{proof}
Take a generator $\tau$ of $I$. We may assume that there is $u\in \Z$ satisfying $0\leq u\leq m$ and $\sigma^{2^{n}}=\tau^{2^{u}}$ (recall the notation that $\sigma$ is a lift of Frobenius). 

\textbf{Case 1.~$n\leq u\leq m$ or $m=u$. }

In this case, (i) occurs under the following. 
\begin{itemize}
\item If $n\leq u\leq m$, then $(1,0)$ and $(0,1)$ correspond to $\tau$, and $\sigma \tau^{-2^{u_1-n}}$ respectively. 
\item If $u=m$, then $(1,0)$ and $(0,1)$ correspond to $\tau_1$ and $\sigma$ respectively. 
\end{itemize}

\textbf{Case 2.~$u<\min\{m,n\}$. }

In this case, (ii) occurs under that $(1,0)$ and $(0,1)$ correspond to $\tau \sigma^{-2^{n-u}}$ and $\tau$ respectively. 
\end{proof}

\begin{prop}\label{abc2}
\emph{Suppose $p=2$ and that there is an isomorphism $I\cong \Z/2^{m}\times \Z/2$, where $m\in \Zpn$. Then $F/F_0$ satisfies one of the following: 
\begin{enumerate}
\item $G\cong \Z/2^{m}\times \Z/2^{n}\times \Z/2$ which induces $I\cong \langle(1,0,0),(0,0,1)\rangle$, 
\item $G\cong \Z/2^{m}\times \Z/2^{n+1}$ which induces $I\cong \langle(1,0),(0,2^{n})\rangle$, 
\item $G\cong \Z/2^{u}\times \Z/2^{m+n-u}\times \Z/2$ which induces $I\cong \langle(1,2^{n-u},0),(0,0,1)\rangle$, where $0\leq u<\min\{m,n\}$,
\item $G\cong \Z/2^{u+1}\times \Z/2^{m+n-u}$ which induces $I\cong \langle(1,2^{n-u}),(2^{u},1)\rangle$, where $0\leq u<\min\{m,n\}$. 
\end{enumerate}}
\end{prop}

\begin{proof}
Fix an isomorphism $I\cong \Z/2^{m}\times \Z/2$, and take elements $\tau_1$ and $\tau_2$ of $I$ that correspond to $(1,0)$ and $(0,1)$ respectively. We may assume that there are $u\in \{0,1,\ldots,m\}$ and $u'\in \{0,1\}$ satisfying $\sigma^{2^{n}}=\tau_1^{2^{u}}\tau_2^{u'}$ (recall the notation that $\sigma$ is a lift of Frobenius). 

\textbf{Case 1.~$n\leq u_1\leq m$ or $u_1=m$. }

First, we consider the case $u'=0$. Then (i) occurs under the following. 
\begin{itemize}
\item If $n\leq u\leq m$, then $(1,0,0)$, $(0,1,0)$ and $(0,0,1)$ correspond to $\tau_1$, $\sigma \tau_{1}^{-2^{u_1-n}}$ and $\tau_2$ respectively. 
\item If $u=m$, then $(1,0,0)$, $(0,1,0)$ and $(0,0,1)$ correspond to $\tau_1$, $\sigma$ and $\tau_2$ respectively. 
\end{itemize}
Second, we consider the case $u_2=1$. Then (ii) occurs under the following: 
\begin{itemize}
\item If $n\leq u\leq m$, then $(1,0)$ and $(0,1)$ correspond to $\tau_1$ and $\sigma \tau_{1}^{-2^{u-n}}$ respectively. 
\item If $u=m$, then $(1,0)$ and $(0,1)$ correspond to $\tau_1$ and $\sigma$ respectively. 
\end{itemize}
Note that $\tau_2$ corresponds to $(0,2^{n})$ in the both cases. 

\textbf{Case 2. $0\leq u<\min\{m,n\}$. }

First, we consider the case $u'=0$. Then (iii) occurs under that $(1,0,0)$, $(0,1,0)$ and $(0,0,1)$ correspond to $\tau_1\sigma^{-2^{n-u}}$, $\sigma$ and $\tau_2$ respectively. 

Second, we consider the case $u'=1$. Then (iv) occurs under that $\tau_1\sigma^{-2^{n-u}}$ and $\sigma$ respectively. Note that $\tau_2$ corresponds to $(2^{u},0)$. 
\end{proof}

\begin{prop}\label{niab}
\emph{
\begin{enumerate}
\item If $F/F_0$ is cyclic, then we have $\Ima(\nu_{F/F^{+},F_0})=\N_{F_2/F_0}(F_2^{\times})$, where $F_2/F_0$ is the unique quadratic subextension of $F/F_0$. 
\item If $F/F_0$ is not cyclic, then $\nu_{F/F^{+},F_0}$ is surjective. 
\end{enumerate}}
\end{prop}

\begin{proof}
(i): Since $(F^{+,\times}:\N_{F/F^{+}}(F^{\times}))=2$, the index of $\Ima(\nu_{F/F^{+},F_0})=\N_{F/F^{+}}(F^{\times})\cap F_0^{\times}$ in $F_0^{\times}$ is less than or equal to $2$. If they are the same, then the local class field theory implies that there is a subgroup $H'$ of $G$ of index $2$ distinct from that of $G$ corresponding to $F_2$, which is absurd by the cyclicity of $G$. Hence we have $(F_0^{\times}:\Ima(\nu_{F/F^{+},F_0}))=2$, which implies
\begin{equation*}
\Ima(\nu_{F/F^{+},F_0})=\N_{F_2/F_0}(F_2^{\times}). 
\end{equation*}

(ii): By replacing $F_0$ to a subfield of $F$ containing $F_0$, we may assume that there is an isomorphsim $G\cong \Z/2\times \Z/2$ such that $F^{+}$ corresponds to $\langle(1,0)\rangle$. Let $F'/F_0$ and $F''/F_0$ be subextensions of $F/F_0$ corresponding to $\langle(0,1)\rangle$ and $\langle(1,1)\rangle$ respectively. Then $\Ima(\nu_{F/F^{+},F_0})$ contains both $\N_{F'/F_0}(F'^{\times})$ and $\N_{F''/F_0}(F''^{\times})$. Since these are distinct subgroups of $F_0^{\times}$ of index $2$ by the local class field theory, we obtain the desired result. 
\end{proof}

\subsection{$R$-equivalence problem}\label{retf}

Let $T$ be a torus over $F_0$. In this section, we consider the question submitted by Colliot-Th{\'e}l{\`e}ne and Suresh \cite[Question Locale]{ColliotThelene2007}. 
\begin{itemize}
\item[($R_{T}$)] Does we have $RT(F_0)\cdot K_{T}=T(F_0)$?
\end{itemize}

There are many known affirmative results on ($R_{T}$). In this paper, we will only use the following: 

\begin{prop}\label{ctsa}
\emph{The question (R$_{T}$) is affirmative if one of the following hold: 
\begin{enumerate}
\item $T$ is induced over $F_0$,
\item $T$ splits over a cyclic extension of $F_0$.
\item $T$ splits over a totally ramified Galois extension of $F_0$. 
\end{enumerate}}
\end{prop}

\begin{proof}
(i), (ii): These follow from Proposition \ref{crtw}. 

(ii): This is \cite[Proposition 2.1 (v)]{ColliotThelene2007}. 
\end{proof}

In the sequel of this section, we consider the question ($R_{T}$) for $T=T_{F/F^{+},F_0}$, which is simply denoted by ($R_{F/F^{+},F_0}$) in the sequel. Moreover, we rewrite $K_{F/F^{+},F_0}$ for $K_{T_{F/F^{+},F_0}}$. 

\begin{cor}\label{spaf}
\emph{If $F=F^{+}\times F^{+}$, then ($R_{F/F^{+},F_0}$) is affirmative. }
\end{cor}

\begin{proof}
This follows from Proposition \ref{trid} (ii). 
\end{proof}

\begin{lem}\label{rdfl}
\emph{Let $F^{+}=F_1^{+}\times \cdots \times F_r^{+}$ and $F=F_1\times \cdots \times F_r$, where $F_i^{+}$ is a field and $F_i$ is an {\'e}tale quadratic algebra over $F_i^{+}$ for any $i$. 
\begin{enumerate}
\item The following diagram is Cartesian: 
\begin{equation*}
\xymatrix{
K_{F/F^{+},F_0}\ar[r] \ar[d]&\prod_{i=1}^{r}K_{F_i/F_i^{+},F_0} \ar[d]\\
T_{F/F^{+},F_0}(F_0)\ar[r] &\prod_{i=1}^{r}T_{F_i/F_i^{+},F_0}(F_0)}
\end{equation*}
\item Assume $F_1^{+}=\cdots=F_r^{+}$ and $F_1=\cdots=F_r$. Then ($R_{F/F^{+},F_0}$) is affirmative if and only if so is for ($R_{F_{1}/F_{1}^{+},F_{0}}$). 
\end{enumerate}}
\end{lem}

\begin{proof}
(i): This follows from the definitions of $K_{F/F^{+},F_0}$ and $K_{F_i/F_i^{+},F_0}$. 

(ii): This is a consequence of (i) and Proposition \ref{n1rt} (i). 
\end{proof}

\emph{Until the end of Section \ref{abng}, we assume that both $F^{+}$ and $F$ are fields}. For an open subgroup $H'$ of $\Gamma_{k_0}$, we denote by $\nu_{k/k^{+},k_0*}\!\mid_{H'}$ the composite map
\begin{equation*}
X_{*}(T_{F/F^{+},F_0})^{H'}\hookrightarrow X_{*}(T_{F/F^{+},F_0}) \rightarrow \Z. 
\end{equation*}

\begin{prop}\label{sjcc}
\emph{The following are equivalent: 
\begin{enumerate}
\item ($R_{F/F^{+},F_0}$) is affirmative and $\val_{F_0}\circ \nu_{F/F^{+},F_0}$ is surjective, 
\item $\nu_{F/F^{+},F_0*}\!\mid_{\Gamma_{F'}}$ is surjective for some totally ramified subextension $F'/F_0$ of $F_0^{\sep}/F_0$. 
\end{enumerate}}
\end{prop}

\begin{proof}
(i) $\Rightarrow$ (ii): Considering the coflasque resolution
\begin{equation*}
1\rightarrow F(T_{F/F^{+},F_0})\rightarrow P(T_{F/F^{+},F_0})\rightarrow T_{F/F^{+},F_0}\rightarrow 1
\end{equation*}
of $T_{F/F^{+},F_0}$ as in Section \ref{cofl}, there is a finite separable extension $F'/F_0$ and a morphism
\begin{equation*}
\varphi \colon \Res_{F'/F_0}\G_m\rightarrow T
\end{equation*}
such that $\val_{F_0}\circ \nu_{F/F^{+},F_0}\circ \varphi$ is surjective. On the other hand, Proposition \ref{frrp} implies $\nu_{F/F^{+},F_0}\circ \varphi=\N_{F'/F_0}^{n'}$ for some $n'\in \Z$. Then the surjectivity implies $n'\in \{\pm 1\}$ and that $F'/F_0$ is totally ramified. In particular, $\nu_{F/F^{+},F_0}\!\mid_{\Gamma_{F'}}$ is surjective. 

(ii) $\Rightarrow$ (i): Take a totally ramified extension $F'/F_0$ and an element $c$ of the preimage of $1$ under $\nu_{F/F^{+},F_0*}\!\mid_{\Gamma_{F'}}$ which satisfy (ii). Let
\begin{equation*}
\psi_{c}\colon \Res_{F'/F_0}\G_m \rightarrow T_{F/F^{+},F_0}
\end{equation*}
be the homomorphism induced by
\begin{equation*}
\Z \rightarrow X_{*}(T_{F/F^{+},F_0}\otimes_{F_0}F');1\mapsto c
\end{equation*}
via Propositions \ref{treq} and \ref{frrp}. Then we have $\nu_{F/F^{+},F_0}\circ \psi_{c}=\N_{F'/F_0}$. Moreover, we have
\begin{equation*}
\val_{F_0}\circ \nu_{F/F^{+},F_0}\circ \psi_{c}=\val_{F_0}\circ \N_{F'/F_0}=[F':F_0]\val_{F'}. 
\end{equation*}
Since $[F':F_0]=1$, we obtain the surjectivity of $\val_{F_0}\circ \nu_{F/F^{+},F_0}\circ \psi_{c}$. This implies (i) as desired. 
\end{proof}

\subsection{Determination of the negative conditions}\label{abng}

Here we keep the assumptions and notations on $F/F_0$ in Section \ref{abex}. We determine $F/F^{+}$ for which ($R_{F/F^{+},F_0}$) is negative. Recall that $G:=\Gal(F/F_0)$, and $I\subset G$ is the inertia group of $F/F_0$. Moreover, let $H^{+}:=\Gal(F/F^{+})$. 

\begin{thm}\label{rpng}
\emph{Assume that the inertia group of $F/F_0$ is cyclic. Then ($R_{F/F^{+},F_0}$) is negative if and only if one of the following hold. 
\begin{enumerate}
\item There is an isomorphism $G\cong \Z/2^{m}\times \Z/2^{n}$ where $0<m<n$ that induces $H^{+}\cong \langle (0,2^{m-1})\rangle$ and $I\cong \langle (1,0)\rangle$ respectively (in particular, $F/F^{+}$ is unramified). 
\item There is an isomorphism $G\cong \Z/2^{u}\times \Z/2^{m+n-u}$ where $0<u<\min\{m,n\}$ that induces $H^{+}\cong \langle (0,2^{m+n-u-1})\rangle$ and $I\cong \langle (1,2^{n-u})\rangle$ respectively (in particular, $F/F^{+}$ is ramified). 
\end{enumerate}
In particular, the map $\nu_{F/F^{+},F_0}$ is surjective on the $F_0$-valued points if ($R_{F/F^{+},F_0}$) is negative. }
\end{thm}

\begin{thm}\label{r2ng}
\emph{Assume $p=2$ and that the inertia group of $F/F_0$ is of the form $\Z/2^{m}\times \Z/2$ where $m\in \Zpn$. Then ($R_{F/F^{+},F_0}$) is negative if and only if one of the following hold. 
\begin{enumerate}
\item There is an isomorphism $G\cong \Z/2^{m}\times \Z/2^{n}\times \Z/2$ where $m<n$ that induces $H^{+}\cong \langle (0,2^{n-1},0)\rangle$ and $I\cong \langle (1,0,0),(0,0,1)\rangle$ respectively (in particular, $F/F^{+}$ is unramified). 
\item There is an isomorphism $G\cong \Z/2^{m}\times \Z/2^{n+1}$ where $m\leq n$ that induces $H^{+}\cong \langle (0,2^{n})\rangle$ and $I\cong \langle (1,0),(0,2^{n})\rangle$ respectively (in particular, $F/F^{+}$ is ramified). 
\item There is an isomorphism $G\cong \Z/2^{u}\times \Z/2^{m+n-u}\times \Z/2$ where $0<u<\min\{m,n\}$ that induces $H^{+}\cong \langle (0,2^{m+n-u-1},0)\rangle$ and $I\cong \langle (1,2^{n-u},0),(0,0,1)\rangle$ respectively (in particular, $F/F^{+}$ is ramified). 
\item There is an isomorphism $G\cong \Z/2^{u+1}\times \Z/2^{m+n-u}$ where $0\leq u<\min\{m,n\}$ that induces $H^{+}\cong \langle (0,2^{m+n-u-1})\rangle$ and $I\cong \langle (1,2^{n-u}),(2^{u},0)\rangle$ respectively (in particular, $F/F^{+}$ is ramified). 
\end{enumerate}
In particular, the map $\nu_{F/F^{+},F_0}$ is surjective if ($R_{F/F^{+},F_0}$) is negative. }
\end{thm}

\begin{rem}
Theorems \ref{rpng} and \ref{r2ng} give a complete determination of all pairs of abelian extensions $F/F^{+}$ of $F_0$ for which ($R_{F/F^{+},F_0}$) is negative in the case $p>2$ or $F_0=\Q_2$. Otherwise, it will be difficult for a complete study since the (wild) inertia group of $F/F_0$ may become more complicated. 
\end{rem}

By Theorem \ref{rpng}, we can derive some positivities of ($R_{F/F^{+},F_0}$):

\begin{cor}\label{m3af}
\emph{The question ($R_{F/F^{+},F_0}$) is affirmative if one of the following hold: 
\begin{enumerate}
\item $F^{+}/F_0$ is unramified, 
\item $[F:F_0]\not\in 8\Z$. 
\end{enumerate}}
\end{cor}

\begin{proof}
These follow from Theorems \ref{rpng} and \ref{r2ng}. 
\end{proof}

We prove Theorems \ref{rpng} and \ref{r2ng} in the sequel. 

\begin{proof}[Proof of Theorem \ref{rpng}]
First, suppose that $G$ satisfies Proposition \ref{abna} (i). We may assume $m,n>0$, since ($R_{F/F^{+},F_0}$) becomes affirmative by Proposition \ref{ctsa} (ii) if $m=0$ or $n=0$. Then, $\nu_{F/F^{+},F_0}$ is surjective on $F_0$-valued points by Proposition \ref{niab} (ii). Moreover, $H^{+}$ is equal to either $\langle(2^{m-1},0)\rangle$, $\langle(0,2^{n-1})\rangle$ or $\langle(2^{m-1},2^{n-1})\rangle$. In the first and the third cases, the subgroup $H'$ of $G$ corresponding to $\langle (0,1)\rangle$ satisfies the surjectivity of $\nu_{F/F^{+},F_0*}\!\mid_{H'}$. Next, consider the second case. Note that $F/F^{+}$ is unramified since $\langle(0,2^{n-1})\rangle$ is not contained in $\langle (1,0)\rangle$. If $m\geq n$, then the subgroup $H'$ of $G$ corresponding to $\langle (1,1)\rangle$ satisfies the surjectivity of $\nu_{F/F^{+},F_0*}\!\mid_{H'}$. Next, we assume $(u=)\,m<n$. Let $H'$ be a subgroup of $G$ such that $H'\cdot \langle (1,0)\rangle=G$. Then $H'$ contains an element of the form $(a,1)$, where $a\in \Z$. However, we have $2^{n-1}(a,1)=(0,2^{n-1})$, which implies that $\nu_{F/F^{+},F_0*}\!\mid_{H'}$ is not surjective. Note that this is the negative condition (i). 

Second, suppose that $G$ satisfies Proposition \ref{abna} (ii). We may assume $u>0$. Indeed, if $u=0$, ($R_{F/F^{+},F_0}$) becomes affirmative by Proposition \ref{ctsa} (ii). Then, $\nu_{F/F^{+},F_0}$ is surjective on $F_0$-valued points by Proposition \ref{niab} (ii). Moreover, $H^{+}$ is equal to either $\langle (2^{u-1},0)\rangle$, $\langle (0,2^{m+n-u-1})\rangle$ or $\langle (2^{u-1},2^{m+n-u-1})\rangle$. In the first and the third cases, the subgroup $H'$ of $G$ corresponding to $\langle (0,1)\rangle$ satisfies the surjectivity of $\nu_{F/F^{+},F_0*}\!\mid_{H'}$. On the other hand, consider the second case. Note that $F/F^{+}$ is ramified since $(0,2^{m+n-u-1})=2^{m-1}(1,2^{n-u})$. Moreover, the assumption on $u$ implies the inequality $u<m+n-u-1$. Now let $H'$ be a subgroup of $G$ such that $H'\cdot \langle (1,2^{n-u})\rangle=G$. Then $H'$ contains an element of the form $(a,1)$, where $a\in \Z$. However, we have $2^{m+n-u-1}(a,1)=(0,2^{m+n-u-1})$, which implies that $\nu_{F/F^{+},F_0*}\!\mid_{H'}$ is not surjective. Note that this is the negative condition (ii). 
\end{proof}

\begin{proof}[Proof of Theorem \ref{r2ng}]
We may assume $n>0$, since ($R_{F/F^{+},F_0}$) becomes affirmative by Proposition \ref{ctsa} (iii) if $n=0$. Moreover, the hypothesis and Proposition \ref{niab} (ii) imply that $\nu_{F/F^{+},F_0}$ is surjective on $F_0$-valued points. Hence it suffices to consider Proposition \ref{sjcc} (ii). 

First, suppose that $G$ satisfes Proposition \ref{abc2} (i). Then the same argument as the proof Theorem \ref{rpng} (i) implies that Proposition \ref{sjcc} (ii) does not hold if and only if $H^{+}=\langle (0,2^{n-1},0) \rangle$ and $(u=)\,m<n$. Note that this condition is identical to (i). 

Second, suppose that $G$ satisfes Proposition \ref{abc2} (ii). Then the same argument as Theorem \ref{rpng} (i) implies that Proposition \ref{sjcc} (ii) does not hold if and only if $H^{+}=\langle (0,2^{n}) \rangle$ and $m\leq n$, which is the condition (ii). 

Third, suppose that $G$ satisfes Proposition \ref{abc2} (iii). Then the same argument as the proof Theorem \ref{rpng} (ii) implies that Proposition \ref{sjcc} (ii) does not hold if and only if $H^{+}=\langle (0,2^{m+n-u_1-1},0) \rangle$. Note that this condition is identical to (iii). 

Finally, suppose that $G$ satisfes Proposition \ref{abc2} (iv). Then the same argument as Theorem \ref{rpng} (ii) implies that Proposition \ref{sjcc} (ii) does not hold if and only if $H^{+}=\langle (0,2^{m+n-u_1-1}) \rangle$, which is the condition (iv). 
\end{proof}

\subsection{Kottwitz maps of tori}\label{kwmp}

For a non-archimedean local field $F$, we denote by $\Fb$ the completion of the maximal unramified extension of $F$ in $F^{\sep}$. Here we recall the Kottwitz maps of tori over $\Fb_0$ defined by \cite[7.1--7.3]{Kottwitz1997}, that are functorial surjective homomorphisms
\begin{equation*}
\kappa_{\breve{T}}\colon \breve{T}(\Fb_0)\rightarrow X_{*}(\breve{T})_{\Gamma_{\Fb_0}}. 
\end{equation*}

\textbf{Case 1.~$\breve{T}$ is induced over $\Fb_0$. }
Suppose that there is an isomorphism of $\Fb_0$-tori
\begin{equation*}
\breve{T}=\Res_{\Fb_1/\Fb_0}\G_m\times \cdots \times \Res_{\Fb_r/\Fb_0}\G_m, 
\end{equation*}
where $\Fb_i$ is a field for any $i$. Then we have $X_{*}(\breve{T})_{\Gamma_{\Fb_0}}\cong \Z^{\oplus r}$. Now we define $\kappa_{\breve{T}}$ as the following diagram becomes commutative: 
\begin{equation*}
\xymatrix{
\breve{T}(\Fb_0)\ar[r]^{\kappa_{\breve{T}}}\ar[d]^{\cong} & X_{*}(\breve{T})_{\Gamma_{\Fb_0}}\ar[d]^{\cong}\\
\prod_{i=1}^{r}\Fb_i^{\times}\ar[r]^{\hspace{2mm}(\val_{\Fb_i})_{i}} & \Z^{\oplus r}. }
\end{equation*}
Here the lower horizontal map is given by $(t_i)_{i}\mapsto (\ord_{\Fb_i}(t_i))_{i}$. Note that the surjectivity of $\kappa_{T}$ follows by definition. 

\textbf{Case 2.~General case. }
Take an induced torus $\breve{P}$ over $\Fb_0$ and a surjective homomorphism $\breve{P}\rightarrow \breve{T}$. Then both $X_{*}(\breve{P})\rightarrow X_{*}(\breve{T})$ and $\breve{P}(\Fb_0)\rightarrow \breve{T}(\Fb_0)$ are surjective. We define $\kappa_{\breve{T}}$ as the following diagram becomes commutative: 
\begin{equation*}
\xymatrix{
\breve{P}(\Fb_0)\ar[r]^{\kappa_{\breve{P}}\hspace{3mm}}\ar[d] & X_{*}(\breve{P})_{\Gamma_{\Fb_0}} \ar[d] \\
\breve{T}(\Fb_0)\ar[r]^{\kappa_{\breve{T}}\hspace{3mm}} & X_{*}(\breve{T})_{\Gamma_{\Fb_0}}. }
\end{equation*}
The well-definedness of $\kappa_{\breve{T}}$ and the functoriality is proved in \cite[7.3]{Kottwitz1997}. 

Now let $T$ be a torus over $F_0$. Recall that $\sigma$ is a lift of the $q$-Frobenius on $F_0$, where $q$ is the cardinality of the residue field of $F_0$. As in \cite[p.300]{Kottwitz1997}, the Kottwitz map $\kappa_{T_{\Fb_0}}$ induces a surjective homomorphism
\begin{equation*}
\kappa_{T}\colon T(F_0)\rightarrow X_{*}(T)_{I_{F_0}}^{\sigma}. 
\end{equation*}

We consider the target of the Kottwitz map of $T_{F/F^{+},F_0}$. Until Section \ref{vrrp}, write
\begin{equation*}
F^{+}=F_{1}^{+}\times \cdots \times F_{r}^{+},\quad F=F_{1}\times \cdots \times F_{r}, 
\end{equation*}
where $F_{j}^{+}$ is a field and $F_{j}$ is an {\'e}tale quadratic algebra over $F_{j}^{+}$. 

\begin{prop}\label{cceq}
\emph{Assume that $F_{i}$ is not a ramified quadratic extension of $F_{i}^{+}$ for all $i$. Then $X_{*}(T_{F/F^{+},F_0})_{I_{F_0}}$ is torsion-free. }
\end{prop}

\begin{proof}
By assumption, we have
\begin{equation*}
T_{F/F^{+},F_0}\otimes_{F_0}\Fb_0\cong \G_m \times \prod_{j=1}^{r}\Res_{(F_{j}^{+}\otimes_{F_0}\Fb_0)/\Fb_0}\G_m, 
\end{equation*}
by Proposition \ref{trid} (ii). In particular, $T_{F/F^{+},F_0}$ is induced over $\Fb_0$. Hence the assertion follows. 
\end{proof}

\begin{dfn}\label{cdar}
We say that $F/F^{+}$ satisfies (r) if it satisfies the two conditions as follows: 
\begin{itemize}
\item $F_1=\cdots =F_r$ and $F_{1}^{+}=\cdots =F_r^{+}$, 
\item $F_1/F_0$ is abelian and $F_1/F_1^{+}$ is ramified. 
\end{itemize}
\end{dfn}

\begin{lem}\label{ts2t}
\emph{Assume that $F/F^{+}$ satisfies (r). 
\begin{enumerate}
\item The abelian group $X_{*}(T_{F/F^{+},F_0}^{1})_{I_{F_0}}$ is a finite direct sum of $\Z/2$. 
\item For any $d\not\in 2\Z$, the multiplication by $d$ on $X_{*}(T_{F/F^{+},F_0})_{I_{F_0}}$ is injective. 
\end{enumerate}}
\end{lem}

\begin{proof}
(i): This follows from the fact that the generator of $\Gal(F_1/F_1^{+})$ acts on $X_{*}(T_{F/F^{+},F_0}^{1})$ by the $(-1)$-multiple. 

(ii): The upper exact sequence in Proposition \ref{trid} (i) induces an exact sequence
\begin{equation*}
X_{*}(T_{F/F^{+},F_0}^{1})_{I_{F_0}}\rightarrow X_{*}(T_{F/F^{+},F_0})_{I_{F_0}}\xrightarrow{\nu_{F/F^{+},F_0*}} \Z \rightarrow 0. 
\end{equation*}
Since $X_{*}(T_{F/F^{+},F_0}^{1})_{I_{F_0}}$ is annihilated by $2$ by (i), the assertion is obtained. 
\end{proof}

\begin{lem}\label{ccr2}
\emph{Assume that $F/F^{+}$ satisfies (r). Let $F'_{1}/F_0$ be a subextension of $F_1/F_0$ such that $F_1/F'_{1}$ is of degree $d\not\in 2\Z$. Set $F'^{+}_1:=F'_1\cap F_1^{+}$, $F':=\prod_{r}F'_{1}$ and $F'^{+}:=\prod_{r}F'^{+}_1$. 
\begin{enumerate}
\item If $F_1/F'_1$ is totally ramified, then the canonical map $i\colon T_{F'/F'^{+}}\hookrightarrow T_{F/F^{+}}$ induces an isomorphism
\begin{equation*}
X_{*}(T_{F'/F'^{+},F_0})_{I_{F_0}} \cong X_{*}(T_{F/F^{+},F_0})_{I_{F_0}}. 
\end{equation*}
\item If $F_1/F'_1$ is unramified, then the canonical map $i\colon T_{F'/F'^{+}}\hookrightarrow T_{F/F^{+}}$ induces an isomorphism
\begin{equation*}
X_{*}(T_{F'/F'^{+},F_0})_{I_{F_0}}^{\sigma} \cong X_{*}(T_{F/F^{+},F_0})_{I_{F_0}}^{\sigma}. 
\end{equation*}
\end{enumerate}}
\end{lem}

\begin{proof}
By definition, the composite
\begin{equation*}
T_{F'/F'^{+},F_0}\xrightarrow{i}T_{F/F^{+},F_0}\xrightarrow{\N_{F/F'}}T_{F'/F'^{+},F_0}
\end{equation*}
is given by $t\mapsto t^{d}$. Hence the map
\begin{equation*}
i_{*}\colon X_{*}(T_{F'/F'^{+},F_0})_{I_{F_0}}\rightarrow X_{*}(T_{F/F^{+},F_0})_{I_{F_0}}
\end{equation*}
is injective by Lemma \ref{ts2t}. Moreover, there is a commutative diagram
\begin{equation*}
\xymatrix{
X_{*}(T_{F'/F'^{+},F_0})_{I_{F_0}}\ar[rd]_{\nu_{F'/F'^{+},F_0*}\hspace{3mm}} \ar[r]^{i_{*}} & X_{*}(T_{F/F^{+},F_0})_{I_{F_0}} \ar[d]^{\nu_{F/F^{+},F_0*}} \\
& \Z. }
\end{equation*}

(i): The map 
\begin{equation*}
T_{F/F^{+},F_0}\xrightarrow{\N_{F/F'}}T_{F'/F'^{+},F_0}\xrightarrow{i}T_{F/F^{+},F_0}
\end{equation*}
induces the multiplication by $d$ on $X_{*}(T_{F/F^{+},F_0})_{I_{F_0}}$, which is injective by Lemma \ref{ts2t}. Hence the map $X_{*}(T_{F'/F'^{+},F_0})_{I_{F_0},\tor}\rightarrow X_{*}(T_{F/F^{+},F_0})_{I_{F_0},\tor}$ induced by $i_{*}$ is surjective. Hence the assertion follows. 

(ii): The map 
\begin{equation*}
T_{F/F^{+},F_0}\xrightarrow{\N_{F/F'}}T_{F'/F'^{+},F_0}\xrightarrow{i}T_{F/F^{+},F_0}
\end{equation*}
induces the multiplication by $d$ on $X_{*}(T_{F/F^{+},F_0})_{I_{F_0}}^{\sigma}$, which is injective by Lemma \ref{ts2t}. Hence the map $X_{*}(T_{F'/F'^{+},F_0})_{I_{F_0},\tor}\rightarrow X_{*}(T_{F/F^{+},F_0})_{I_{F_0},\tor}$ induced by $i_{*}$ is surjective. On the other hand, the assumption $d\not\in 2\Z$ implies that $\nu_{F/F^{+},F_0*}\colon X_{*}(T_{F/F^{+},F_0})_{I_{F_0}}^{\sigma}\rightarrow \Z$ is surjective if and only if so is for $\nu_{F'/F'^{+},F_0*}\colon X_{*}(T_{F'/F'^{+},F_0})_{I_{F_0}}^{\sigma}\rightarrow \Z$. Therefore the assertion follows. 
\end{proof}

In the sequel of Section \ref{vrrp}, we use the notations in Section \ref{abex} for $F_1/F_0$ if its degree is a power of $2$. Moreover, for $N\in \Zpn$, we denote by $\ebar_{N,1}\ldots,\ebar_{N,N}$ the standard basis of $(\Z/2)^{\oplus N}$ as an $\F_2$-vector space. 

\begin{prop}\label{orbp}
\emph{Assume that $F/F^{+}$ satisfies (r), $[F_1:F_0]$ is a power of $2$ and that $I=\Gal(F_1/F_1^{\ur})$ is cyclic. Write $[F_1^{\ur}:F_0]=2^{n}$ where $n\in \Znn$. Then there is a commutative diagram
\begin{equation*}
\xymatrix{
X_{*}(T_{F/F^{+},F_0}^{1})_{I_{F_0}} \ar[r] \ar[d]^{\cong}& X_{*}(T_{F/F^{+},F_0})_{I_{F_0}}\ar[d]^{\cong} \\
(\Z/2)^{\oplus 2^{n}r}\ar[r]^{f_{n,r}\hspace{7mm}}& \Z \oplus (\Z/2)^{\oplus 2^{n}r-1}, }
\end{equation*}
where $f_{n,r}$ is given by
\begin{equation*}
\ebar_{2^{n}r,i}\mapsto
\begin{cases}
(0,\ebar_{2^{n}r-1,i})&\text{if }i\neq fr,\\
(0,\ebar_{2^{n}r-1,1}+\cdots+\ebar_{2^{n}r-1,2^{n}r-1})&\text{otherwise}. 
\end{cases}
\end{equation*}
Moreover, $\sigma$ acts on the lower left-hand side $(\Z/2)^{\oplus 2^{n}r}$ as follows: 
\begin{equation*}
\ebar_{2^{n}r,i}\mapsto
\begin{cases}
\ebar_{2^{n}r,i+r} &\text{if }1\leq i\leq (2^{n}-1)r,\\
\ebar_{2^{n}r,i-(2^{n}-1)g} &\text{otherwise}. 
\end{cases}
\end{equation*}}
\end{prop}

\begin{proof}
Take a generator $\tau$ of $I$, and assume $\#I=2^{m}$ where $m\in \Znn$. Then we have
\begin{gather*}
X_{*}(T_{F/F^{+},F_0}^{1})=\bigoplus_{i=1}^{r}\bigoplus_{\nu=0}^{2^{n}-1}\bigoplus_{\mu=0}^{2^{m-1}-1}\Z (\delta_{ij}\sigma^{\nu}\tau^{\mu}(1-\tau^{2^{m-1}}))_{1\leq j \leq r},\\
X_{*}(T_{F/F^{+},F_0})=X_{*}(T_{F/F^{+},F_0}^{1}) \oplus \Z\left(\sum_{i=0}^{2^{n}-1}\sum_{j=0}^{2^{m-1}-1}\sigma^{i}\tau^{j}\right)_{1\leq i\leq r}
\end{gather*}
(recall that $\sigma \in G=\Gal(F_1/F_0)$ is a lift of the Frobenius). Now we define a homomorphism
\begin{equation*}
c_{n,r}^{1}\colon X_{*}(T_{F/F^{+},F_0}^{1})\rightarrow (\Z/2)^{\oplus fr}
\end{equation*}
by sending $(\delta_{ij}\sigma^{\nu}\tau^{\mu}(1-\tau^{2^{m-1}}))_{1\leq j\leq r}$ to $e_{2^{n}r,(\nu-1)r+i}$. Moreover, define a homomorphism
\begin{equation*}
c_{n,r}\colon X_{*}(T_{F/F^{+},F_0})\rightarrow \Z \oplus (\Z/2)^{\oplus fr-1}
\end{equation*}
which sends $\sum_{i=0}^{2^{n}-1}\sum_{j=0}^{2^{m-1}-1}\sigma^{i}\tau^{j}$ to $(1,0)$ and as the following diagram becomes commutative: 
\begin{equation*}
\xymatrix{
X_{*}(T_{F/F^{+},F_0}^{1}) \ar[r] \ar[d]_{c_{n,r}^{1}}& X_{*}(T_{F/F^{+},F_0})\ar[d]^{c_{n,r}} \\
(\Z/2)^{\oplus 2^{n}r}\ar[r]^{f_{n,r}\hspace{7mm}}& \Z \oplus (\Z/2)^{\oplus 2^{n}r-1}. }
\end{equation*}
Then $c_{n,r}^{1}$ and $c_{n,r}$ induce the desired commutative diagram. 
\end{proof}

\subsection{A variant of the $R$-equivalence problem}\label{vrrp}

Set $K_{T}^{\circ}:=\Ker(\kappa_{T})$, which is a compact open subgroup of $T(\Qp)$ by \cite[Note (1)]{Rapoport2005}. We consider the following question: 
\begin{itemize}
\item[($R_{T}^{\circ}$)] Does we have $RT(\Qp)\cdot K_{T}^{\circ}=T(\Qp)$?
\end{itemize}

Since $K_{T}^{\circ}\subset K_{T}$, ($R_{T}$) is affirmative if so is for ($R_{T}^{\circ}$). 

\begin{lem}\label{rckt}
\emph{The following are equivalent: 
\begin{enumerate}
\item ($R_{T}^{\circ})$ is affirmative,
\item $\kappa_{T}(RT(\Qp))=X_{*}(T)_{I_{F_0}}^{\sigma}$, 
\item $X_{*}(P)_{I_{F_0}}^{\sigma}\rightarrow X_{*}(T)_{I_{F_0}}^{\sigma}$ is surjective for any (or some) coflasque resolution of $T$:
\begin{equation*}
1\rightarrow F\rightarrow P\rightarrow T\rightarrow 1. 
\end{equation*}
\end{enumerate}}
\end{lem}

\begin{proof}
(i)$\Leftrightarrow$(ii): This follows by the definition of $K_{T}^{\circ}$. 

(ii)$\Rightarrow$(iii): It is a consequence of the definition of $RT(\Qp)$ and the functoriality of the Kottwitz maps. 

(iii)$\Rightarrow$(i): Let $t\in T(F_0)$, and take $a\in X_{*}(P)_{I_{F_0}}^{\sigma}$ which maps to $\kappa_{T}(t)$. Note that it is possible by hypothesis. Then the surjectivity of $\kappa_{P}$ implies that there is $t'\in P(F_0)$ such that $\kappa_{P}(t')=a$. We denote by $t_0\in T(F_0)$ the image of $t'$, which is contained in $RT(F_0)$. Then $t_0^{-1}t$ is contained in $K_{T}^{\circ}$. In summary, we have $t=t_0(t_0^{-1}t)\in RT(F_0)\cdot K_{T}^{\circ}$ as desired. 
\end{proof}

\begin{prop}\label{trca}
\emph{The question ($R_{T}^{\circ}$) is affirmative if $T$ splits over a totally ramified Galois extension of $F_0$. }
\end{prop}

\begin{proof}
Fix a splitting field $\widetilde{F}$ of $T$ which is totally ramified Galois over $F_0$. Take a coflasque resolution
\begin{equation*}
1\rightarrow F\rightarrow P\rightarrow T\rightarrow 1
\end{equation*}
of $T$. We may assume that both $F$ and $P$ split over $\widetilde{F}$. Then the homomorphism
\begin{equation*}
X_{*}(P)_{I_{F_0}}\rightarrow X_{*}(T)_{I_{F_0}}
\end{equation*}
is surjective. Since $\sigma$ acts trivially on the both-hand sides, Lemma \ref{rckt} implies the desired assertion. 
\end{proof}

Now we consider ($R_{T}^{\circ}$) for $T=T_{F/F^{+},F_0}$, which will be rewritten as ($R_{F/F^{+},F_0}^{\circ}$). 

\begin{prop}\label{urcp}
\emph{If $F_i$ is not a ramified quadratic extension of $F_i^{+}$ for all $i$, then we have $K_{F/F^{+},F_0}=K_{F/F^{+},F_0}^{\circ}$. In particular, ($R_{F/F^{+},F_0}^{\circ}$) is affirmative if and only if so is for ($R_{F/F^{+},F_0}$). }
\end{prop}

\begin{proof}
This follows from Proposition \ref{cceq}. 
\end{proof}

\begin{thm}\label{rcrm}
\emph{Assume that $F/F^{+}$ satisfies the same hypothesis as Proposition \ref{orbp}. Write $[F_1:F_1^{\ur}]=2^m$ and $[F_1^{\ur}:F_0]=2^{n}$, where $m,n\in \Znn$. 
If ($R_{F/F^{+},F_0}$) is affirmative, then so is for ($R_{F/F^{+},F_0}^{\circ}$) if and only if $\Gal(F/F_0)$ is cyclic or $m\leq n$. Otherwise, $\kappa_{T_{F/F^{+},F_0}}(RT_{F/F^{+},F_0}(F_0))_{\tor}$ is equal to the image of the map $X_{*}(T_{F/F^{+},F_0}^{1})_{I_{F_0}}^{\sigma}\rightarrow X_{*}(T_{F/F^{+},F_0})_{I_{F_0}}^{\sigma}$, and we have
\begin{align*}
X_{*}(T_{F/F^{+},F_0})_{I_{F_0}}^{\sigma}/\kappa_{T_{F/F^{+},F_0}}(RT_{F/F^{+},F_0}(F_0))&\cong X_{*}(T_{F/F^{+},F_0})_{I_{F_0},\tor}^{\sigma}/\kappa_{T_{F/F^{+},F_0}}(RT_{F/F^{+},F_0}(F_0))_{\tor}\\
&\cong \Z/2. 
\end{align*}}
\end{thm}

\begin{proof}
By Proposition \ref{abna}, Lemma \ref{rdfl} (ii) and Theorem \ref{rpng}, we have one of the following: 
\begin{itemize}
\item[(a)] $\Gal(F_1/F_0)$ is cyclic,
\item[(b)] $\Gal(F_1/F_0)$ satisfies the condition (i) in Proposition \ref{abna} for $m,n>0$. 
\end{itemize}
If (a) is satisfied, then ($R_{F/F^{+},F_0}^{\circ}$) is affirmative by Proposition \ref{crtw} (ii). Now suppose that (b) is satisfied. In the sequel of the proof, we use the notations and the identification 
\begin{equation*}
X_{*}(T_{F/F^{+},F_0})_{I_{F_0}}\cong \Z\oplus (\Z/2)^{\oplus 2^{n}r-1}
\end{equation*}
induced by $c_{n,r}$, as in the proof of Proposition \ref{orbp}. Then $\sigma$ fixes $(1,0)$. We denote by $\tau$ and $\sigma$ the elements of $G$ corresponding to $(1,0)$ and $(0,1)$ respectively. 

First, the homomorphism
\begin{equation*}
\Res_{F_1^{\sigma}/F_0}\G_m \rightarrow T_{F/F^{+},F_0};x\mapsto \left(\prod_{i=0}^{2^{m-1}-1}\tau^i(x),\ldots,\prod_{i=0}^{2^{m-1}-1}\tau^i(x)\right)
\end{equation*}
is a direct summand of $P(T_{F/F^{+},F_0})\rightarrow T_{F/F^{+},F_0}$, which induces the following diagram:
\begin{equation*}
\xymatrix{
X_{*}(\Res_{F_1^{\sigma}/F_0}\G_m)_{I_{F_0}} \ar[r] \ar[d]^{\cong}& X_{*}(T_{F/F^{+},F_0})_{I_{F_0}}\ar[d]^{c_{n,r}} \\
\Z \ar[r] & \Z \oplus (\Z/2)^{\oplus 2^{n}r-1}, }
\end{equation*}
Here the lower horizontal map is given by $1\mapsto (1,0)$. Hence the image of $X_{*}(P(T_{F/F^{+},F_0}))_{I_{F_0}}^{\sigma}\rightarrow X_{*}(T_{F/F^{+},F_0})_{I_{F_0}}^{\sigma}$ contains $(1,0)$ since $F_1^{\sigma}/F_0$ is totally ramified. This implies the isomorphism
\begin{equation*}
X_{*}(T_{F/F^{+},F_0})_{I_{F_0}}^{\sigma}/\kappa_{T_{F/F^{+},F_0}}(RT_{F/F^{+},F_0}(F_0))\cong X_{*}(T_{F/F^{+},F_0})_{I_{F_0},\tor}^{\sigma}/\kappa_{T_{F/F^{+},F_0}}(RT_{F/F^{+},F_0}(F_0))_{\tor}. 
\end{equation*}

Second, Proposition \ref{n1rt} (ii), (iii) imply that the image $M$ of
\begin{equation*}
X_{*}(T_{F/F^{+},F_0}^{1})_{I_{F_0}}^{\sigma}\rightarrow X_{*}(T_{F/F^{+},F_0})_{I_{F_0}}^{\sigma}
\end{equation*}
is contained in that of $X_{*}(P(T_{F/F^{+},F_0}))_{I_{F_0}}^{\sigma}\rightarrow X_{*}(T_{F/F^{+},F_0})_{I_{F_0}}^{\sigma}$, that is, $\kappa_{T_{F/F^{+},F_0}}(RT_{F/F^{+},F_0}(F_0))$. Note that Proposition \ref{orbp} implies the equality
\begin{equation*}
M=\{(0,(a_1,\ldots,a_{2^{n}r-1}))\in X_{*}(T_{F/F^{+},F_0}^{1})_{I_{F_0}}^{\sigma}\mid a_{r}=0\}. 
\end{equation*}

\textbf{Case 1.~$m\leq n$. }
The homomorphism
\begin{equation*}
\Res_{F_1^{\sigma\tau}/F_0}\G_m \rightarrow T_{F/F^{+},F_0};x\mapsto \left(\prod_{i=0}^{2^{m-1}-1}\tau^i(x),\ldots,\prod_{i=0}^{2^{m-1}-1}\tau^i(x)\right)
\end{equation*}
is a direct summand of $P(T_{F/F^{+},F_0})\rightarrow T_{F/F^{+},F_0}$. Moreover, it induces the commutative diagram
\begin{equation*}
\xymatrix{
X_{*}(\Res_{F_1^{\sigma\tau}/F_0}\G_m)_{I_{F_0}} \ar[r] \ar[d]^{\cong}& X_{*}(T_{F/F^{+},F_0})_{I_{F_0}}\ar[d]^{c_{n,r}} \\
\Z \ar[r] & \Z \oplus (\Z/2)^{\oplus 2^{n}r-1}, }
\end{equation*}
where the lower horizontal map is given by
\begin{equation*}
1\mapsto \left(1,\sum_{j=0}^{2^{n-1}-1}(\ebar_{2^{n}r-1,2jr+1}+\cdots+\ebar_{2^{n}r-1,2j(r+1)})\right). 
\end{equation*}
This induces $X_{*}(P(T_{F/F^{+},F_0}))_{I_{F_0}}^{\sigma}\rightarrow X_{*}(T_{F/F^{+},F_0})_{I_{F_0}}^{\sigma}$ is surjective since $F_1^{\sigma\tau}/F_0$ is totally ramified. Combining this result with Lemma \ref{rckt}, we obtain that ($R_{F/F^{+},F_0}^{\circ}$) is affirmative. 

\textbf{Case 2.~$m>n$. }
Let $H$ be a subgroup of $\Gal(F_1/F_0)$. If $H$ contains $\tau^{2^{m-1}}$, then a homomorphism $\Res_{F_1^{H}/F_0}\G_m\rightarrow T_{F/F^{+},F_0}$ which is a direct summand of $P(T_{F/F^{+},F_0})\rightarrow T_{F/F^{+},F_0}$ is of the form
\begin{equation*}
\Res_{F_1^{H}/F_0}\G_m \rightarrow T_{F/F^{+},F_0};t\mapsto \N_{F_1^{H}/F_0}(t)^{s},
\end{equation*}
where $s\in \{0,\pm 1\}$. Hence the induced map $X_{*}(\Res_{F_1^{H}/F_0}\G_m)_{I_{F_0}}^{\sigma}\rightarrow X_{*}(T_{F/F^{+},F_0})_{I_{F_0}}^{\sigma}$ is given by $1\mapsto (2s',0)$, where $s\mid s'$. Otherwise, suppose that the image of the second projection
\begin{equation*}
\Gal(F_1/F_0)\cong \Z/2^{m}\times \Z/2^{n}\rightarrow \Z/2^{n}. 
\end{equation*}
is equal to $\langle 2^{k}\rangle$, where $0\leq k \leq n$. Then $H$ is trivial if $k=n$, and otherwise $H$ is of the form $\langle \sigma^{2^{k}}\tau^{N}\rangle$ where $2^{m-n+k}\mid N$. Hence any direct summand $\Res_{F_1^{H}/F_0}\G_m\rightarrow T_{F/F^{+},F_0}$ of $P(T_{F/F^{+},F_0})\rightarrow T_{F/F^{+},F_0}$ induces the homomorphism $\Z\rightarrow X_{*}(T_{F/F^{+},F_0})_{I_{F_0}}^{\sigma}$ whose image is contained in the subgroup of $X_{*}(T_{F/F^{+},F_0})_{I_{F_0}}^{\sigma}$ generated by $(2^{k},0)$ and $M$. Therefore we obtain the latter isomorphism. 
\end{proof}

\section{The question of Bruhat--Colliot-Th{\'e}l{\`e}ne--Sansuc--Tits}\label{warf}

In this section, for a number field $L'$, we write $\In(L')$ and $\Ram(L')$ for the sets of prime numbers that inert and ramify in $L'$ respectively. Moreover, if $L'/\Q$ is abelian, then for a prime number $\ell$, we denote by $D_{\ell}(L'/\Q)$ and $I_{\ell}(L'/\Q)$ the decomposition and the inertia groups of $L'/\Q$ at $\ell$ respectively. 

\subsection{General theory}

Let $T$ be a torus over $\Q$. Here we consider the following: 
\begin{itemize}
\item[(A$_{T,p}$)\hspace{-4mm}] \hspace{4mm}Does we have $T(\Q)\cdot K_{T_{\Qp}}=T_{L/L^{+}}(\Qp)$?
\item[(A$_{T,p}^{\circ}$)\hspace{-4mm}] \hspace{4mm}Does we have $T(\Q)\cdot K_{T_{\Qp}}^{\circ}=T_{L/L^{+}}(\Qp)$?
\end{itemize}
By definition, (A$_{T,p}$) is affirmative if so is for (A$_{T,p}^{\circ}$). 

The following important result will be used later. 

\begin{prop}\label{r2wa}
\emph{
\begin{enumerate}
\item The question (A$_{T,p}$) is affirmative if so is for ($R_{T_{\Qp}}$). 
\item The image of the composite
\begin{equation*}
T(\Q)\hookrightarrow T(\Qp)\xrightarrow{\kappa_{T_{\Qp}}}X_{*}(T)_{I_{\Qp}}^{\sigma}
\end{equation*}
contains $\kappa_{T_{\Qp}}(RT(\Qp))$. In particular, (A$_{T,p}^{\circ}$) is affirmative if so is for ($R_{T_{\Qp}}^{\circ}$). 
\end{enumerate}}
\end{prop}

\begin{proof}
The assertion (i) is \cite[Proposition 2.2]{ColliotThelene2007}. The proof of (ii) is the same as (i). See \cite[p.279]{ColliotThelene2007}. 
\end{proof}

Let $L$ be a CM field, and denote by $L^{+}$ the maximally totally real subfield of $L$. From now on, we simply write $T_{L/L^{+}}$ for the $\Q$-torus $T_{L/L^{+},\Q}$ defined in Definition \ref{tkkd}. For a prime number $\ell$, we set $K_{L/L^{+},\ell}:=K_{T_{L/L^{+}},\ell}$. On the other hand, we define $K_{L/L^{+},p}^{\circ}:=K_{T_{L/L^{+}},p}^{\circ}$. Then rewrite (A$_{L/L^{+}}$) and (A$_{L/L^{+}}^{\circ}$) for (A$_{T_{L/L^{+}},p}$) and (A$_{T_{L/L^{+}},p}^{\circ}$) respectively. 

\begin{lem}\label{rdtp}
\emph{Let $L_1/\Q$ be a subextension of $L/\Q$ such that $[L:L_1]$ is an odd number. Assume that $L_1^{+}:=L_1\cap L^{+}$ satisfies $[L_1:L_1^{+}]=2$. 
\begin{enumerate}
\item The question (A$_{L/L^{+},p}$) is affirmative if and only if so is for (A$_{L_1/L_1^{+},p}$). 
\item Assume that $L_{\Qp}/L_{\Qp}^{+}$ satisfies (r) (see Defintion \ref{cdar}), and that the places of $L$ lying above $v_1$ is unique for any place $v_1\mid p$ of $L_1$. Then (A$_{L/L^{+},p}^{\circ}$) is affirmative if and only if so is for (A$_{L_1/L_1^{+},p}^{\circ}$). 
\end{enumerate}}
\end{lem}

\begin{proof}
(i): By Proposition \ref{bmrt}, both $\Ima(\ord_{p}\circ \nu_{L/L^{+},\Q})$ and $\Ima(\ord_{p}\circ \nu_{L_1/L_1^{+},\Q})$ contains $2\Z$. Hence the assertion is a consequence of the following commutative diagram: 
\begin{equation*}
\xymatrix@C=64pt{
T_{L/L^{+}}(\Q)\ar[d]^{\N_{L/L_1}}\ar[r]^{\hspace{7mm}\ord_{p}\circ \nu_{L/L^{+},\Q}}& \Z \ar[d]^{[L:L_1]\times}\\
T_{L_1/L_1^{+}}(\Q)\ar[r]^{\hspace{7mm}\ord_{p}\circ \nu_{L_1/L_1^{+},\Q}}& \Z. }
\end{equation*}

(ii): By Lemma \ref{ccr2}, the canonical injection $T_{L_1/L_1^{+}}\hookrightarrow T_{L/L^{+}}$ induces an isomorphism
\begin{equation*}
X_{*}(T_{L_1/L_1^{+}})_{I_{\Qp}}^{\sigma}\cong X_{*}(T_{L/L^{+}})_{I_{\Qp}}^{\sigma}. 
\end{equation*}
Hence the assertion follows from Lemma \ref{ts2t} and the functorality of the Kottwitz maps. 
\end{proof}

\begin{lem}\label{rmqd}
\emph{Assume $p>2$ and $L$ contains a quadratic field that is ramified at $p$. Then (A$_{L/L^{+},p}$) is affirmative. }
\end{lem}

\begin{proof}
Since $p>2$, there is $a\in L'$ so that $\N_{L/L^{+}}(a)\in p\cdot \Z_{(p)}^{\times}$. This implies $a\in T_{L/L^{+}}(\Q)$ and the positivity of (A$_{L/L^{+},p}$). 
\end{proof}

In the sequel, we denote by $S(L/L^{+})$ the set of prime numbers $\ell$ such that $\nu_{L_{\Ql}/L_{\Ql}^{+},\Ql}(K_{L/L^{+},\ell})$ is \emph{strictly} contained in $\Z_{\ell}^{\times}$. 

\begin{lem}\label{sndi}
\emph{Assume that $L/\Q$ is an abelian extension. Then the set $S(L/L^{+})$ is the set of prime numbers $\ell'$ such that $L/L^{+}$ is ramified at all $v\mid p$ and $D_{\ell}(L/\Q)=I_{\ell}(L/\Q)$. }
\end{lem}

\begin{proof}
If $D_{\ell}(L/\Q)$ is not cyclic, then Proposition \ref{niab} implies the surjectivity of $\nu_{L_{\Ql}/L_{\Ql}^{+},\Ql}$. Otherwise, we have $\Ima(\nu_{L_{\Ql}/L_{\Ql}^{+},\Ql})=\N_{F/\Ql}(F^{\times})$ where $F/\Ql$ is the unique quadratic extension with Galois group $D_{\ell}(L/\Q)/2$. Then we have $\nu_{L_{\Ql}/L_{\Ql}^{+},\Ql}(K_{L/L^{+},\ell})\subsetneq \Z_{\ell}^{\times}$ if and only if $F/\Ql$ is ramified. Hence the assertion follows. 
\end{proof}

\subsection{Proofs of the affirmative theorems}

Here we give a proof of the main theorems on the positivities of (A$_{L/L^{+},p}$) and of (A$_{L/L^{+},p}^{\circ}$). 

\begin{thm}\label{mtaa}
\emph{Suppose that $L$ is an abelian extension of $\Q$ and $L/L^{+}$ is unramified at all $v\mid p$. Then both (A$_{L/L^{+},p}$) and (A$_{L/L^{+},p}^{\circ}$) are affirmative if one of the following hold: 
\begin{enumerate}
\item $L/L^{+}$ is split at all $v\mid p$,
\item the ramification index of $L^{+}/\Q$ at $p$ is an odd number,
\item $p>2$ and $[L:\Q]\not\in 32\Z$,
\item $p=2$ and $[L:\Q]\not\in 8\Z$. 
\end{enumerate}}
\end{thm}

\begin{proof}
It suffices to prove that (A$_{L/L^{+},p}$) is affirmative by Proposition \ref{urcp}. We may assume $[L:\Q]\in 2^{\Znn}$ by Lemma \ref{rdtp} (i). Then we have one of the following: 
\begin{itemize}
\item[(a)] ($R_{L_{\Qp}/L_{\Qp}^{+},\Qp}$) is affirmative,
\item[(b)] $p>2$ and $L$ contains a quadratic field which is ramified at $p$,
\item[(c)] $p>2$, $\Gal(L/\Q)\cong \Z/4\times \Z/4$ and $\nu_{L/L^{+},\Q}\colon T_{L/L^{+}}(\Qp)\rightarrow \Qpt$ is surjective. 
\end{itemize}
Note that (a) contains the cases (i), (ii) and (iv). Indeed, it follows by Corollary \ref{spaf} if $L/L^{+}$ satisfies (i), and by Corollary \ref{m3af} and Lemma \ref{rdfl} (ii) otherwise. We prove the assertion for each cases. 

(a): In this case, the assertion follows from Proposition \ref{r2wa} (i). 

(b): This is a consequence of Lemma \ref{rmqd}. 

(c): We may assume that $\Gal(L/\Q)\cong \Z/4\times \Z/4$ induces $\Gal(L/L^{+})\cong \langle (2,0)\rangle$. We denote by $L'$ the subfield of $L$ corresponding to $\langle(1,0),(0,2)\rangle$. 

\begin{claim}
The set $S(L/L^{+})\cap (\In(L')\cup \Ram(L'))$ is empty. 
\end{claim}

Indeed, if $\ell \in S(L/\Q)$, then Lemma \ref{sndi} implies $D_{\ell}(L/\Q)=I_{\ell}(L/\Q)$, and it is either $\langle (2,0)\rangle$, $\langle (1,0)\rangle$ or $\langle (1,2)\rangle$. Hence $L'/\Q$ is split at $\ell$. 

Let us consider the following set: 
\begin{equation*}
Q:=\{\ell \in \In(L'/\Q)\mid p\ell \in \N_{L_{\ell'}/L_{\ell'}^{+}}(L_{\ell'}^{\times})\text{ for all }\ell' \in S(L/L^{+})\}. 
\end{equation*}
Then $Q$ is determined as non-empty conditions on modulo the least common multiple of $8$ and $\Delta \cdot \prod_{\ell \in S(L/\Q)}\ell$, where $\Delta$ is the discriminant of $L'/\Q$. Note that we use Claim for the non-emptiness. Hence the Dirichlet's prime number theorem implies that $Q$ is an infinite set. Take $\ell \in Q$. We may assume that $\ell$ is unramified in $L/\Q$. Then the $\ell$-Frobenius on $L/\Q$ is of the form $(a,\pm 1)$ where $a\in \Z/4$ by the assumption that $\ell$ inerts in $L'$. Hence $L/L^{+}$ splits at all $v\mid \ell$ since $\Gal(L/L^{+})\cong \langle (2,0)\rangle$. This implies $p\ell \in \N_{L_{\Ql}/L_{\Ql}^{\times}}(L_{\Ql}^{\times})$. Combining this result with the surjectivity of $\nu_{L_{\Qp},L_{\Qp}^{+},\Qp}$, the first condition on $Q$ and the assumption that $L/\Q$ is abelian, we obtain $p\ell \in \N_{L_{v}/L_{v}^{+}}(L_{v}^{\times})$ for all places $v$ of $L^{+}$. Therefore $p\ell \in \N_{L/L^{+}}(L^{\times})$ by the Hasse norm principle for the cyclic extension $L/L^{+}$. 
\end{proof}

\begin{thm}\label{mtas}
\emph{Suppose that $L$ is an abelian extension of $\Q$ and $L/L^{+}$ is ramified at all $v\mid p$. Then both (A$_{L/L^{+},p}$) and (A$_{L/L^{+},p}^{\circ}$) are affirmative if $L$ satisfies at least one of the conditions (ii)--(iv) in Theorem \ref{mtaa}. }
\end{thm}

\begin{proof}
It suffices to prove that (A$_{L/L^{+},p}^{\circ}$) is affirmative. By Lemma \ref{rdtp} (ii), we may assume that the orders of $D_{p}(L/\Q)$ and $I_{p}(L/\Q))$ are powers of $2$. Moreover, Proposition \ref{trca} and Theorem \ref{rcrm} imply that we have one of the following: 
\begin{itemize}
\item[(a)] ($R_{L_{\Qp}/L_{\Qp}^{+},\Qp}^{\circ}$) is affirmative, 
\item[(b)] $p>2$ and $D_{p}(L/\Q)\cong \Z/2^{m}\times \Z/2^{n}$ which induces $I_{p}(L/\Q)\cong \langle(1,0)\rangle$, where $m>n>0$. 
\end{itemize}
Note that (a) contains the case (iv). In the case (a), the assertion follows from Proposition \ref{r2wa} (ii). In the sequel of the proof, we assume (b). Then we have
\begin{equation*}
X_{*}(T_{L/L^{+}})_{I_{\Qp}}^{\sigma}/\kappa_{T_{L/L^{+}}}(RT_{L/L^{+}}(\Qp))\cong \Z/2. 
\end{equation*}
Let $L'$ be the subfield of $L$ corresponding to the subgroup $\langle (2^{m-1},1) \rangle$ of $D_{p}(L/\Q)\cong \Z/2^{m}\times \Z/2^{n}$, and put $L'^{+}:=L'\cap L^{+}$. Then $L'/L'^{+}$ is quadratic which is ramified at all places above $p$. Moreover, we have $D_{p}(L'/\Q)=I_{p}(L'/\Q)$ and it is cyclic. Hence there is $a\in T_{L'/L'^{+}}(\Q)$ satisfying $\ord_{p}(\N_{L'/L'^{+}}(a))=1$. 

\begin{claim}
The element $\kappa_{T_{L/L^{+},\Qp}}(a)$ is not contained in $\kappa_{T_{L/L^{+},\Qp}}(RT_{L/L^{+}}(\Qp))$. 
\end{claim}

If Claim is proved, then Lemma \ref{rdtp} (ii) implies the surjectivity of the map
\begin{equation*}
T_{L/L^{+}}(\Q)\hookrightarrow T_{L/L^{+}}(\Qp)\xrightarrow{\kappa_{T_{L/L^{+}}}}X_{*}(T_{L/L^{+}})_{I_{\Qp}}^{\sigma}, 
\end{equation*}
which means that (A$_{L/L^{+},p}^{\circ}$) is affirmative. 

For a proof of the Claim, write $L_{\Qp}=\prod_{r}F_1$ and $L_{\Qp}^{+}=\prod_{r}F_1^{+}$, where $F_1^{+}$ is a field and $F_1/F_1^{+}$ is ramified quadratic. In the sequel, we use the same notations as in the proof of Theorem \ref{rcrm} (ii). Let $\iota_{1},\ldots \iota_{r}$ be embeddings $L\hookrightarrow F_1$ which attach all places $v\mid p$ of $L$. Moreover, choose an embedding $\iota \colon F_1\hookrightarrow \Fb_1$, which commutes with the actions of $I_{F_0}$. Then there is an isomorphism
\begin{equation*}
L\otimes_{\Q}\Qpb \xrightarrow{\cong} \prod_{2^{n}r}\Fb_1;\,x\mapsto (\iota \circ \sigma^{\nu}\circ \iota_{1}(x),\ldots ,\iota \circ \sigma^{\nu}\circ \iota_{r}(x))_{0\leq \nu \leq 2^{n}-1}. 
\end{equation*}
We regard $T_{L/L^{+}}(\Qpb)$ as a subgroup of $\prod_{2^{n}r}\Fb_1^{\times}$ under the above isomorphism. Then we have
\begin{equation*}
\kappa_{T_{L/L^{+},\Qp}}(a)=\kappa_{T_{L/L^{+},\Qpb}}(\iota \circ \iota_{1}(a),\ldots,\iota \circ \iota_{r}(a),\iota \circ \tau^{2^{m-1}}\circ \iota_{1}(a),\ldots,\iota \circ \tau^{2^{m-1}}\circ \iota_{r}(a))_{2^{n-1}}. 
\end{equation*}
by $\sigma(a)=\tau^{2^{m-1}}(a)$. Since the norm map $\N_{\Fb_1/\Qpb}$ is surjective, there is $x\in \Fb_1^{\times}$ satisfying $\N_{\Fb_1/\Qpb}(x)=\N_{L/L^{+}}(a)$. On the other hand, the homomorphism $\varphi_0$ in the proof of Theorem \ref{rcrm} (ii) gives the equality
\begin{equation*}
\kappa_{T_{L/L^{+},\Qpb}}(x\tau(x)\cdots \tau^{2^{m-1}-1}(x),\ldots,x\tau(x)\cdots \tau^{2^{m-1}-1}(x))=(1,0). 
\end{equation*}
Moreover, $(u\tau^{2^{m-1}}(u)^{-1},\ldots,u\tau^{2^{m-1}}(u)^{-1})$ is contained in the kernel of $\kappa_{T_{L/L^{+},\Qpb}}$ for $u\in \Fb_1^{\times}$. Hence we obtain the equality
\begin{equation*}
\kappa_{T_{L/L^{+},\Qpb}}(\iota \circ \iota_1(a),\ldots,\iota \circ \iota_1(a))=(1,0). 
\end{equation*}
This implies that $\kappa_{T_{L/L^{+},\Qpb}}(a\cdot (\iota \circ \iota_1(a),\ldots,\iota \circ \iota_1(a))^{-1})$ is contained in $X_{*}(T_{L/L^{+}})_{I_{\Qp},\tor}^{\sigma}$. However, it is not contained in $M$ since $\ord_{p}(\N_{L/L^{+}}(a))=1$ (here we use $p>2$). 
\end{proof}

\subsection{Construction of negative examples}\label{nexc}

Here we prove the following: 

\begin{thm}\label{ngmt}
\emph{
\begin{enumerate}
\item Assume $p>2$. For $d\in 32\Z$, there is an infinite family $\{L_j\}_{j\in J}$ of CM fields of degree $d$ that are abelian over $\Q$ such that (A$_{L_j/L_j^{+},p}$) is negative for any $j\in J$. Moreover, if $p\equiv 1\bmod 4$, then both the sets
\begin{equation*}
J_{\ur}:=\{j\in J\mid L_j/L_j^{+}\text{ is unramified at all }v\mid p\}
\end{equation*}
and $J\setminus J_{\ur}$ are infinite sets. 
\item For $d\in 8\Z$, there is an infinite family $\{L_j\}_{j\in J'}$ of CM fields of degree $d$ that are abelian over $\Q$ such that (A$_{L_j/L_j^{+},2}$) is negative for any $j\in J'$. Moreover, both the sets
\begin{equation*}
J'_{\ur}:=\{j\in J'\mid L_j/L_j^{+}\text{ is unramified at all }v\mid p\}
\end{equation*}
and $J'\setminus J'_{\ur}$ are infinite. 
\end{enumerate}}
\end{thm}

We give a proof of Theorem \ref{ngmt} in the sequel. We will use a sufficient condition for the negativity of (A$_{L/L^{+},p}$) as follows: 

\begin{prop}\label{ngax}
\emph{Let $\ell_0$ be a prime number satisfying $\ell_0 \equiv 1\bmod 4$. Suppose that $L$ is an abelian extension of $\Q$ which satisfies the following: 
\begin{enumerate}
\item there is an isomorphism
\begin{equation*}
\Gal(L/\Q)\cong \Z/2^{m}\times H, 
\end{equation*}
where $m\geq 2$ and $H$ is a $2^{m-1}$-torsion group, which induces isomomorphisms
\begin{equation*}
\Gal(L/L^{+})\cong \langle (2^{m-1},0)\rangle,\quad D_{\ell_0}(L/\Q)=I_{\ell_0}(L/\Q)\cong \Z/2^{m}\times \{0\}
\end{equation*}
(in particular, $\ell_0$ is contained in $S(L/L^{+})$),
\item the subfield of $L$ corresponding to $\langle 2\rangle \times H$ is $\Q(\sqrt{\ell_0})$,
\item $\ell \in \N_{L/L^{+}}(L^{\times})$ for all $\ell \in \Ram(L)\setminus \{p,\ell_0\}$,
\item $p\in \N_{L_{\Qp}/L_{\Qp}^{+}}(L_{\Qp}^{\times})$ and $p\not\in \N_{L_{\Q_{\ell_0}}/L_{\Q_{\ell_0}}^{+}}(L_{\Q_{\ell_0}}^{\times})$. 
\end{enumerate}
Then both (A$_{L/L^{+},p}$) and (A$_{L/L^{+},p}^{\circ}$) negative. }
\end{prop}

\begin{proof}
It suffices prove that (A$_{L/L^{+},p}$) is negative. Let $pa\in \N_{L_{\Q_{\ell_0}}/L_{\Q_{\ell_0}}^{+}}(L_{\Q_{\ell_0}}^{\times})$ where $a\in \Z_{(p)}^{\times}$. Write $a=\ell_1^{m_1}\cdots \ell_{s}^{m_s}$, where $\ell_i\neq p$ is a prime number and $m_i\in \Z$ for any $i$. By Proposition \ref{bmrt}, we may assume $m_i=1$ for all $i$. Moreover, we may assume $\ell_i \not\in \Ram(L)$ for all $i$ by (ii), (iii). On the other hand, the second condition in (iv) implies that there is $1\leq i\leq s$ so that $\ell_{i}\not\in \N_{L_{\Q_{\ell'}}/L_{\Q_{\ell'}}^{+}}(L_{\Q_{\ell'}}^{\times})$. By (ii), the $\ell_i$-Frobenius on $L/\Q$ is of the form $(a,b)$, where $a$ is a generator of $\Z/2^{n_1}$ and $b\in H$. Hence $\ell_i$ inerts in $L/L^{+}$ by (i), which implies $\ell_i\not\in \ell_i^{2\Z}\times \Z_{\ell_i}^{\times}=\N_{L_{\Q_{\ell_i}}/L_{\Q_{\ell_i}}^{+}}(L_{\Q_{\ell_i}}^{\times})$. Therefore we obtain $pa\not\in \N_{L_{\Q_{\ell_i}}/L_{\Q_{\ell_i}}^{+}}(L_{\Q_{\ell_i}}^{\times})$, and hence $pa\not\in \N_{L/L^{+}}(L^{\times})$. 
\end{proof}

We construct $L$ as in Proposition \ref{ngax} in the sequel. First, we assume $p>2$. Take $m\geq 3$. Let $\ell$ be a prime number satisfying the following: 
\begin{equation*}
\left(\frac{\ell}{p}\right)=1,\quad \ell \equiv 
\begin{cases}
1\bmod 8 &\text{if }p\equiv 1\bmod 4,\\
5\bmod 8 &\text{if }p\equiv -1\bmod 4. 
\end{cases}
\end{equation*}
We denote by $L_{(\ell),0}/\Q$ the unique subextension of $\Q(\zeta_{\ell})$ of degree $4$. Then
\begin{equation*}
L_{(\ell),1}:=L_{(\ell),0}\left(\sqrt{(-1)^{(p-1)/2}p}\right)
\end{equation*}
is abelian over $\Q$. Moreover, there is an isomorphism
\begin{equation*}
\Gal(L_{(\ell),0}/\Q)\cong \Z/4\times \Z/2
\end{equation*}
whose inertia groups at $p$ and at $\ell$ correspond to $\{0\}\times \Z/2$ and $\Z/4\times \{0\}$ respectively. Let $L_{(\ell),2}/\Q$ be the subextension of $L_{(\ell),1}/\Q$ corresponding to $\langle (2,1)\rangle$. Then $L_{(\ell),2}$ is totally real. Moreover, the ramification index and the residue degree of $L_{(\ell),1}$ at $p$ are $2$ and $1$ respectively. Next, take a prime number $\ell'$ satisfying $\ell'\equiv 2^{m}+1\bmod 2^{m+1}$ and the $\ell'$-Frobenius in $\Gal(L_{(\ell),1}/\Q)$ corresponds to $(2,1)$. Note that such a prime number exists by the Chebotarev density theorem. Then $L_{(\ell),2}/\Q$ is totally split at $\ell'$ and
\begin{equation*}
\left(\frac{p}{\ell'}\right)=-1. 
\end{equation*}
We denote by $L_{(\ell'),m,0}/\Q$ the unique subextension of $\Q(\zeta_{\ell'})/\Q$ of degree $2^{m}$. Then, the assumption on $\ell'$ implies that $L_{(\ell'),m,0}$ is a CM field. Finally, put
\begin{equation*}
L_{(\ell,\ell'),m,0}:=L_{(\ell),2}L_{(\ell'),m,0}. 
\end{equation*}
We denote by $J_{m,0}$ the set of pairs $(\ell,\ell')$ where $\ell$ and $\ell'$ are as above. 

\begin{prop}\label{nex1}
\emph{Let $p>2$ and $m\geq 3$. For any $(\ell,\ell')\in J_{m,0}$, $L_{(\ell,\ell'),m,0}$ is an abelian extension of $\Q$ of degree $2^{m+2}$ which satisfies (i)--(iv) in Propsosition \ref{ngax} for $\ell_0=\ell'$. }
\end{prop}

\begin{proof}
By construction, we have
\begin{equation*}
\Gal(L_{(\ell,\ell'),m,0}/\Q)\cong \Gal(L_{(\ell'),m,0}/\Q)\times \Gal(L_{(\ell),2}/\Q)\cong \Z/2^{m}\times \Z/4,
\end{equation*}
and $I_{\ell'}(L_{(\ell,\ell'),m,0}/\Q)$ corresponds to $\Z/2^{m}\times \{0\}$. Hence (i) is true. Moreover, (ii) and (iii) hold by construction. On the other hand, since $D_{p}(L_{(\ell,\ell'),m,0}/\Q)\cong \langle (1,0),(0,2) \rangle$ by construction, we obtain the positivity of the first condition in (iv). Furthermore, the second condition in (iv) follows from $\left(\frac{p}{\ell'}\right)=-1$. 
\end{proof}

\begin{rem}
The extensions $L_{(\ell,\ell'),m,0}/L_{(\ell,\ell'),m,0}^{+}/\Q$ induce the case (i) in Theorem \ref{rpng}. 
\end{rem}

Second, we assume $p\equiv 1\bmod 4$. Take $m\geq 3$, and let $\ell$, $L_{(\ell),1}$ and $L_{(\ell),2}$ be as above. We denote by $L'_0/\Q$ the unique subextension of $\Q(\zeta_{p})/\Q$ of degree $4$. Next, take a prime number $\ell'$ such that the $\ell'$-Frobenius in $\Gal(L_{(\ell),2}/\Q)$ corresponds to $\langle(2,1)\rangle$ and 
\begin{equation*}
\ell'\equiv 
\begin{cases}
2^{m}+1&\bmod \,2^{m+1}\text{ if }p\equiv 1\bmod 8,\\
1&\bmod \,2^{m+1}\text{ if }p\equiv 5\bmod 8.
\end{cases}
\end{equation*}
Let $L_{(\ell'),m,0}/\Q$ be the unique subextensions of $\Q(\zeta_{\ell'})/\Q$ of degree $2^{m}$. Then there is an isomorphism
\begin{equation*}
\Gal(L_{(\ell'),m,0}L'_0/\Q)\cong \Z/8\times \Z/4
\end{equation*}
whose inertia groups at $p$ and at $\ell'$ are $\{0\}\times \Z/4$ and $\Z/8\times \{0\}$ respectively. We define $L_{(\ell'),m,1}/\Q$ as the subextensions of $L_{(\ell'),m,0}L'_0/\Q$ that correspond to $\langle(2,1)\rangle$. Then $L_{(\ell'),m,1}/\Q$ is a CM field. Finally, put
\begin{equation*}
L_{(\ell,\ell'),m,1}:=L_{(\ell),2}L_{(\ell'),m,1}. 
\end{equation*}
We denote by $J_{m,1}$ the set of pairs $(\ell,\ell')$ where $\ell$ and $\ell'$ are as above. 

\begin{prop}\label{nex2}
\emph{Let $p\equiv 1\bmod 4$ and $m\geq 3$. For any $(\ell,\ell')\in J_{m,1}$, $L_{(\ell,\ell'),m,1}$ is an abelian extension of $\Q$ of degree $2^{m+2}$ which satisfies (i)--(iv) in Proposition \ref{ngax} for $\ell_0=\ell'$. }
\end{prop}

\begin{proof}
The proof is the same as Proposition \ref{nex1}. 
\end{proof}

\begin{rem}
The extensions $L_{(\ell,\ell'),m,1}/L_{(\ell,\ell'),m,1}^{+}/\Q$ induce the case (ii) in Theorem \ref{rpng}. 
\end{rem}

Finally, we assume $p=2$. For $m\in \Znn$, we denote by $J_{m,2}$ the set of the triples $(\ell',\ell,\{\ell_1,\ldots,\ell_{m}\})$, where $\ell'$ is a prime number satisfying $\ell'\equiv 5\bmod 8$, $\ell$ is a prime number satisfying $\ell \equiv -1\bmod 4$ and $\left(\frac{\ell}{\ell'}\right)=1$, and $\ell_1,\ldots.\ell_{m}$ are prime numbers satisfying $\ell_i\equiv 1\bmod 8$ and $\left(\frac{\ell_i}{\ell'}\right)=1$ for all $i$. For $j=(\ell',\ell,\{\ell_1,\ldots.\ell_{m}\})\in J_{m,2}$, we write $L_{(\ell'),0}$ for the unique subextension of $\Q(\zeta_{\ell'})/\Q$ of degree $4$, and put 
\begin{equation*}
L_{j,0}:=L_{(\ell'),0}(\sqrt{\ell},\sqrt{\ell_1},\ldots,\sqrt{\ell_{m}}). 
\end{equation*}

\begin{prop}\label{wn21}
\emph{Let $m\in \Znn$. For any $j\in J_{m,2}$, $L_{j,0}$ is an abelian extension of degree $2^{m+2}$ which satisfies (i)--(iv) in Proposition \ref{ngax} for $p=2$ and $\ell_0=\ell'$. }
\end{prop}

\begin{proof}
Write $j=(\ell',\ell,\{\ell_1,\ldots,\ell_m\})\in J_{m,2}$. By construction, we have
\begin{equation*}
\Gal(L_{j,0}/\Q)\cong \Gal(L_{(\ell'),0}/\Q)\times \Gal(\Q(\sqrt{\ell})/\Q)\times \prod_{i=1}^{m}\Gal(\Q(\sqrt{\ell_i})/\Q)\cong \Z/4\times \Z/2 \times \prod_{i=1}^{m}\Z/2,
\end{equation*}
which induces isomorphisms
\begin{equation*}
D_{2}(L_{j,0}/\Q)\cong \Z/4\times \Z/2\times \prod_{i=1}^{m}\{0\},\quad I_{\ell'}(L/\Q)\cong \Z/4 \times \{0\}\times \prod_{i=1}^{m}\{0\}. 
\end{equation*}
Hence (i) and the first condition in (iv) hold. Moreover, we have (ii) and (iii) by construction. On the other hand, the second condition in (iv) is affirmative since $\ell \equiv -1\bmod 4$. 
\end{proof}

\begin{rem}
The extensions $L_{j,0}/L_{j,0}^{+}/\Q$ induce the case (i) in Theorem \ref{rpng}. 
\end{rem}

We give an another negative example for (A$_{L/L^{+},2}$). For $m\in \Znn$ and $j=(\ell',\ell,\{\ell_1,\ldots,\ell_m\})$, let $L_{(\ell'),0}$ be as above. Consider $L_{(\ell'),1}:=L_{(\ell'),0}(\sqrt{2})$. Then there is an isomorphism
\begin{equation*}
\Gal(L_{(\ell'),1}/\Q)\cong \Z/4\times \Z/2
\end{equation*}
which induces isomorphisms $I_{\ell'}(L_{(\ell'),1}/\Q)\cong \Z/4\times \{0\}$ and $I_{2}(L_{(\ell'),1}/\Q)\cong \{0\}\times \Z/2$. We write $L_{(\ell'),2}/\Q$ for the subextension of $L_{(\ell'),1}/\Q$ corresponding to $\langle(2,1)\rangle$, and put
\begin{equation*}
L_{j,1}:=L_{(\ell'),2}(\sqrt{\ell},\sqrt{\ell_1},\ldots,\sqrt{\ell_{m}}). 
\end{equation*}

\begin{prop}\label{wn22}
\emph{Let $m\in \Znn$. Then $L_{j,1}$ satisfies (i)--(iv) in Proposition \ref{ngax} for $p=2$ for any $j\in J_{m,2}$. }
\end{prop}

\begin{proof}
The proof is the same as Proposition \ref{wn21}. 
\end{proof}

\begin{rem}
The extensions $L_{j,1}/L_{j,1}^{+}/\Q$ induce the case (ii) in Theorem \ref{r2ng}. 
\end{rem}

\begin{proof}[Proof of Theorem \ref{ngmt}]
We may assume that $d$ is a power of $2$ by Lemma \ref{rdtp}. 

(i): Write $d=2^{m+2}$ where $m\geq 3$. If $p\equiv 1\bmod 4$, then $\{L_{j,m,0}\}_{j\in J_{m,0}}\sqcup \{L_{j',m,1}\}_{j'\in J_{m,1}}$ gives a desired family by Proposition \ref{nex1}. Otherwise, $\{L_{j,m,0}\}_{j\in J_{m,0}}$ is a desired family by Proposition \ref{nex2}. 

(ii): Write $d=2^{m+3}$ where $m\in \Znn$. Then Propositions \ref{wn21} and \ref{wn22} imply that the family $\{L_{j,0},L_{j,1}\}_{j\in J_{m,2}}$ satisfies the desired properties. 
\end{proof}

\section{Connected components of Shimura varieties for CM unitary groups}\label{ccmt}

\subsection{Unitary similitude groups}\label{unsm}

Let $k_0$ be a field of characteristic not equal $2$, $k^{+}/k_0$ a finite separable extension and $k$ an {\'e}tale quadratic algebra over $k^{+}$. We denote by $a\mapsto \overline{a}$ the non-trivial Galois automorphism of $k$ over $k^{+}$. For a $k/k^{+}$-hermitian space, we mean a finite free $k$-module $V$ of finite rank equipped with a $k$-valued bilinear form $\langle \,,\,\rangle$ satisfying 
\begin{equation*}
\langle x,y\rangle=\overline{\langle y,x\rangle},\quad \langle cx,y\rangle=c\langle x,y\rangle
\end{equation*}
for any $c\in k$ and $x,y\in V$. 

Take an element $\delta \in L^{\times}$ satisfying $\overline{\delta}=-\delta$. Let
\begin{equation*}
(\,,\,):=\tr_{k/k_0}\delta^{-1}\langle\,,\,\rangle. 
\end{equation*}
Then we define an algebraic group $G_{V}$ over $k_0$ as
\begin{equation*}
G_{V}(R)=\{(g,c)\in \GL_{k\otimes_{k_0}R}(V\otimes_{k_0}R)\times \G_m(R)\mid (gx,gy)=c(x,y)\text{ for all }x,y\in V\otimes_{k_0}k\}
\end{equation*}
for any $\Q$-algebra $R$. Note that $G_{V}$ is reductive and connected. We denote by $\det_{V}\colon G_{V}\rightarrow \Res_{L/\Q}\G_m$ and $\sml_{V}\colon G_{V}\rightarrow \G_m$ the determinant map of $V$ over $L$ and the similitude character of $G_{V}$ respectively. 

\emph{In the sequel, we assume that $n:=\rk_{k}(V)$ is an odd number. }

\begin{lem}\label{cmtr}
\emph{
\begin{enumerate}
\item Let $Z_{V}$ be the center of $G_{V}$. Then the homomorphism
\begin{equation*}
T_{k/k^{+},k_0}\rightarrow G_{V};t\mapsto t\cdot \id_{V}
\end{equation*}
induces an isomorphism $T_{k/k^{+},k_0}\cong Z_{V}$. 
\item The derived group $G_{V}^{\der}$ of $G_{V}$ is simply connected. 
\item The homomorphism
\begin{equation*}
\delta_{V}\colon G_{V}\rightarrow \Res_{k/k_0}\G_m;(g,t)\mapsto \det\!{}_{V}(g)/\sml_{V}(g)^{(n-1)/2}
\end{equation*}
induces an isomorphism $G_{V}/G_{V}^{\der}\cong T_{k/k^{+},k_0}$. 
\item The following diagram is commutative: 
\begin{equation*}
\xymatrix{
G_{V}\ar[rd]_{\sml_{V}}\ar[r]^{\delta_{V}\hspace{5mm}} & T_{k/k^{+},k_0} \ar[d]^{\nu_{k/k^{+},k_0}} \\
& \G_m. }
\end{equation*}
\end{enumerate}}
\end{lem}

\begin{proof}
The assertion (i) follows from the definition of $G_{V}$. The assertions (ii) and (iii) are contained in \cite[\S 7, pp.393--394]{Kottwitz1992}. Moreover, the equality which defines $G_{V}$ in $\Res_{k/k_0}\GL(V)$ gives (iv). 
\end{proof}

Assume that we have
\begin{equation*}
k^{+}=k_1^{+}\times \cdots \times k_r^{+},\quad k=k_1\times \cdots \times k_r, 
\end{equation*}
where $k_i^{+}$ is a field and $k_i$ is an {\'e}tale quadratic algebra over $k_i^{+}$. Then it induces a factorization
\begin{equation*}
V=V_1\oplus \cdots \oplus V_r, 
\end{equation*}
where $V_i$ is an $k_i/k_i^{+}$-hermitian space for any $i$.

\begin{lem}\label{smfp}
\emph{The following diagram is Cartesian: 
\begin{equation*}
\xymatrix@C=46pt{
G_{V}\ar[r]^{\sml_{V}} \ar[d] & \G_m \ar[d] \\
\prod_{i=1}^{r}G_{V_i} \ar[r]^{\hspace{3mm}(\sml_{V_i})_{i}\hspace{5mm}} & \prod_{i=1}^{r}\G_m. }
\end{equation*}}
\end{lem}

\begin{proof}
This follows from the definitions of $G_{V}$ and $G_{V_i}$. 
\end{proof}

Next, suppose $k_0=\Qp$. We rewrite $F_i^{+}$ and $F_i$ for $k_i^{+}$ and $k_i$ respectively. 

\begin{prop}\label{jcwz}
\emph{Assume both $F^{+}$ and $F$ are fields. Then there is $a\in F^{+,\times}$ and an $F$-basis $e_1,\ldots,e_n$ of $V$ whose Gram matrix of $\langle\,,\,\rangle$ with respect to $e_1,\ldots,e_n$ is
\begin{equation*}
\begin{pmatrix}
&&J\\
&a&\\
J&&
\end{pmatrix}. 
\end{equation*}
Here $J$ is the anti-diagonal matrix of size $(n-1)/2$ that has $1$ at every non-zero entry. }
\end{prop}

\begin{proof}
This follows from \cite[Theorem 3.1, 1)]{Jacobowitz1962}. 
\end{proof}

\begin{dfn}
A \emph{Bruhat--Tits subgroup} of $G_{V}(\Qp)$ is the \emph{full} stabilizer of a self-dual multichain of lattices in $V_{\Qp}$ with respect to the alternating form $(\,,\,)$ in the sense of \cite[Definitions 3.4, 3.13]{Rapoport1996b}. 
\end{dfn}

We give a typical example of a self-dual multichain of lattices in $V_{\Qp}$. Let $1\leq i\leq r$. 

\textbf{Case 1.~$F_{i}=F_{i}^{+}\times F_{i}^{+}$. }
The hypothesis induces a decomposition
\begin{equation*}
V_{i}=V_{i,0}\oplus V_{i,1}. 
\end{equation*}
Consider the $F_i^{+}$-valued alternating form $\tr_{F_i/F_i^{+}}\delta^{-1}\langle\,,\,\rangle$. Then it induces an isomorphism between $V_{i,1}$ and the dual space of $V_{i,0}$. Now take an $F_i^{+}$-basis $e_{i,1},\ldots,e_{i,n}$ of $V_{i,0}$, and denote by $e_{i,1}^{*},\ldots,e_{i,n}^{*}$ its dual basis. Then, for $0\leq j \leq n$, let
\begin{equation*}
\Lambda_{i,0,j}:=\bigoplus_{j\leq j_0} O_{F_i^{+}}\varpi_i e_{i,j}\oplus \bigoplus_{j>j_0}O_{F_i^{+}}e_{i,j}, \quad \Lambda_{i,1,j}:=\bigoplus_{j\leq j_1} O_{F_i^{+}}\varpi_i e_{i,j}^{*}\oplus \bigoplus_{j>j_1}O_{F_i^{+}}e_{i,j}^{*},
\end{equation*}
where $\varpi_{i}$ be a uniformizer of $F_{i}^{+}$. 

\textbf{Case 2.~$F_{i}$ is a field. }
Take $a_i\in F_i^{+,\times}$ and an $F_i$-basis $e_{i,1},\ldots,e_{i,n}$ of $V_{i}$ as in Proposition \ref{jcwz}. For $1\leq j\leq n$, set
\begin{equation*}
\Lambda_{i,j}:=\bigoplus_{j'\leq j} O_{F_i^{+}}\varpi_i e_{i,j'}\oplus \bigoplus_{j'>j}O_{F_i^{+}}e_{i,j'},
\end{equation*}
where $\varpi_{i}$ is a uniformizer of $F_i$. 

By using the above notations, we define a multichain of lattices $\mathcal{L}_{V}$ in $V$ as the set of $\bigoplus_{i=1}^{r}\Lambda_{i}$, where $\Lambda_{i}$ is of the form
\begin{equation*}
\begin{cases}
\varpi_{i}^{m_0}\Lambda_{i,0,j_0}\oplus \varpi_i^{m_1}\Lambda_{i,1,j_1} &\text{if }F_i=F_i^{+}\times F_i^{+}, \\
\varpi_{i}^{m}\Lambda_{i,j} &\text{otherwise. }
\end{cases}
\end{equation*}
Then all self-dual multichain of lattices in $V_{\Qp}$ can be regarded as a subset of $\mathcal{L}_{V}$. 

\begin{rem}
The notions of dual lattices in $V$ with respect to $\langle\,,\,\rangle$ and $(\,,\,)$ may differ. However, $\mathcal{L}_{V}$ is stable under taking the dual lattices with respect to $\langle\,,\,\rangle$. 
\end{rem}

\begin{lem}\label{hsod}
\emph{Let $K_p$ be a Bruhat--Tits subgroup of $G_{V}(\Qp)$. Then we have
\begin{equation*}
\nu_{F/F^{+},\Qp}(K_p)=K_{F/F^{+},\Qp}. 
\end{equation*}}
\end{lem}

\begin{proof}
By Lemmas \ref{rdfl} (i) and \ref{smfp}, we may assume that $F^{+}$ is a field. Moreover, we may assume that $K_p$ is the stabilizer of $\mathcal{L}_{V}$. 

\textbf{Case 1.~$F=F^{+}\times F^{+}$. }
Let $V_0$ and $e_{1},\ldots,e_{n}$ be the objects $V_{i,0}$ and $e_{i,1},\ldots,e_{i,n}$ respectively in the case $F=F_i$. Then there is a commutative diagram: 
\begin{equation*}
\xymatrix@C=46pt{
G_{V} \ar[r]^{\delta_{V}}\ar[d]^{\cong} & T_{F/F^{+},\Qp} \ar[d]^{\cong} \\
\G_m \times \Res_{F^{+}/\Qp}\GL(V_0)\ar[r]^{\hspace{2mm}\id \times \det_{F^{+}}}& \G_m\times \Res_{F^{+}/\Qp}\G_m. }
\end{equation*}
We regard $V$ as the canonical $F^{+}$-vector space $(F^{+})^{\oplus n}$ by $e_1,\ldots,e_n$. Then, for $(t,u)\in \Zpt \times O_{F^{+}}^{\times}=K_{F/F^{+},\Qp}$, we have $g:=(t,\diag(u,1,\ldots,1))\in K_p$ and $\delta_{V}(g)=(t,u)$. Hence the assertion is true. 

\textbf{Case 2.~$F$ is a field. }
Take $a\in F^{+,\times}$ and an $F$-basis $e_1,\ldots,e_n$ of $V$ as in Proposition \ref{jcwz}. Then, for $t\in K_{F/F^{+},\Qp}$, we have $g:=\diag(1,\ldots,1,t,\N_{F/F^{+}}(t),\ldots,\N_{F/F^{+}}(t))\in K_{p}$ and $\delta_{V}(g)=t$. Therefore the assertion follows. 
\end{proof}

Let $K_p$ be a parahoric subgroup of $G_{V}(\Qp)$. It can be written as the intersection of a Bruhat--Tits subgroup $K'_p$ and the kernel of the Kottwitz map $\kappa_{G_{V}}$ of $G_{V,\Qpb}$. Note that $\kappa_{G_{V}}$ is defined as follows (here we use Lemma \ref{cmtr} (ii)): 
\begin{equation*}
\xymatrix@C=46pt{
G_{V}(\Qpb)\ar[r]^{\kappa_{G_{V}}} \ar[d]^{\delta_{V}} & \pi_{1}(G_{V})_{I_{\Qp}} \ar[d]^{\cong} \\
T_{F/F^{+},\Qp}(\Qpb) \ar[r]^{\kappa_{T_{F/F^{+},F_0}}\hspace{3mm}} & X_{*}(T_{F/F^{+},\Qp})_{I_{\Qp}}. }
\end{equation*}
Here $\pi_{1}(G_{V})$ is the Borovoi's algebraic fundamental group of $G_{V}$. Note that the vertical isomorphism is a consequence of \cite[Example 1.6 (2)]{Borovoi1998}. Moreover, the index of $K_p$ in $K'_p$ is finite. See \cite[Appendix, Proposition 3]{Pappas2008}. 

\begin{lem}\label{prhr}
\emph{Let $K_p$ be a parahoric subgroup of $G_{V}(\Qp)$. Then we have
\begin{equation*}
\nu_{F/F^{+},\Qp}(K_p)=K_{F/F^{+},\Qp}^{\circ}. 
\end{equation*}}
\end{lem}

\begin{proof}
This follows from Lemma \ref{hsod} and the description of parahoric subgroups of $G_{V}(\Qp)$. 
\end{proof}

\subsection{Proofs of the main theorems}

Let $n\in \Zpn$ be an odd number. In the sequel, we set $k:=L$ and $k^{+}:=L^{+}$, where $L$ and $L^{+}$ are as in Section \ref{warf}. Moreover, let $a\mapsto \overline{a}$ and $\delta \in L^{\times}$ be the same objects as Section \ref{unsm}. Fix a subset $S$ of $\Hom(L,\C)$ such that the restriction to $L^{+}$ induces an isomorphism $S\cong \Hom(L^{+},\R)$ under the natural surjection $\Hom(L,\C)\rightarrow \Hom(L^{+},\R)$. For $\varphi \in S$, we denote by $\overline{\varphi}$ the composite of the complex conjugation and $\varphi$. Now suppose that $V$ is of signature $\{(r_{\varphi},r_{\overline{\varphi}})\}_{\varphi \in S}$ where $n=r_{\varphi}+r_{\overline{\varphi}}$, and fix a $\C$-basis of $V\otimes_{L^{+},\varphi}\R$ whose Gram matrix is
\begin{equation*}
\diag(1^{(r_{\varphi})},-1^{(r_{\overline{\varphi}})}). 
\end{equation*}
We define $X$ as the $G(\R)$-conjugacy class of the homomorphism
\begin{equation*}
\bS\rightarrow G(\R);z\mapsto (\diag(z^{(r_{\varphi})},\overline{z}^{(r_{\overline{\varphi}})}))_{\varphi \in S}. 
\end{equation*}
Then $(G_{V},X_{V})$ is a Shimura datum, and hence we can consider the Shimura varieties for $(G_{V},X_{V})$. 

Now let $K_p$ be a parahoric or a Bruhat--Tits subgroup of $G_{V}(\Qp)$. For a compact open subgroup $K^p$ of $G_{V}(\A_f^p)$, the Shimura variety for $(G_{V},X_{V})$ with level $K^pK_p$ is defined as follows: 
\begin{equation*}
\Sh_{K^{p}K_p}(G_{V},X_{V}):=G_{V}(\Q)\backslash X_{V}\times G_{V}(\A_{f})/K^{p}K_p. 
\end{equation*}
Moreover, let
\begin{equation*}
\pi_{0}(\Sh_{K_p}(G_{V},X_{V})):=\varprojlim_{K^p}\pi_{0}(\Sh_{K^pK_p}(G_{V},X_{V})). 
\end{equation*}
Then there is a right action on $G_{V}(\A_{f}^{p})$ on $\pi_{0}(\Sh_{K_p}(G_{V},X_{V}))$. We consider the following: 
\begin{itemize}
\item[(T$_{V,K_p}$)] Is the action of $G_{V}(\A_{f}^{p})$ on $\pi_{0}(\Sh_{K_p}(G_{V},X_{V}))$ transitive?
\end{itemize}

Now we interpret (T$_{V,K_p}$) by means of (A$_{L/L^{+},p}$) or (A$_{L/L^{+},p}^{\circ}$). 

\begin{thm}\label{cntr}
\emph{
\begin{enumerate}
\item If $K_p$ is a Bruhat--Tits subgroup of $G_{V}(\Qp)$, then there is an isomorphism of groups
\begin{equation*}
\pi_0(\Sh_{K_p}(G_{V},X_{V}))/G_{V}(\A_f^p)\xrightarrow{\cong}T_{L/L^{+}}(\Q)\backslash T_{L/L^{+}}(\Qp)/K_{L/L^{+},p}.
\end{equation*}
\item If $K_p$ is a parahoric subgroup of $G_{V}(\Qp)$, then there is an isomorphism of groups
\begin{equation*}
\pi_0(\Sh_{K_p}(G_{V},X_{V}))/G_{V}(\A_f^p)\xrightarrow{\cong} T_{L/L^{+}}(\Q)\backslash T_{L/L^{+}}(\Qp)/K_{L/L^{+},p}^{\circ}.
\end{equation*}
\end{enumerate}}
\end{thm}

\begin{proof}
By Lemma \ref{cmtr} (ii), \cite[Theorem 5.17]{Milne} (see also \cite[2.1.3]{Deligne1979}) implies that there is an isomorphism
\begin{equation*}
\pi_0(\Sh_{K_p}(G_{V},X_{V}))\cong T_{L/L^{+}}(\Q)^{\dagger}\backslash T_{L/L^{+}}(\Qp)/\delta_{V}(K_p),
\end{equation*}
where $T_{L/L^{+}}(\Q)^{\dagger}$ is the intersection of $T_{L/L^{+}}(\Q)$ and the image of the canonical homomorphism
\begin{equation*}
Z_{V}(\R)\hookrightarrow G_{V}(\R)\rightarrow T_{L/L^{+}}(\R),
\end{equation*}
which is described as $t\mapsto t\cdot (t/\overline{t})^{(n-1)/2}$. Since it is surjective, we obtain the equality
\begin{equation*}
T_{L/L^{+}}(\Q)^{\dagger}=T_{L/L^{+}}(\Q). 
\end{equation*}
On the other hand, Lemmas \ref{hsod} and \ref{prhr} imply
\begin{equation*}
\delta_{V}(K_p)=
\begin{cases}
K_{L/L^{+},p} &\text{if $K_p$ is Bruhat--Tits},\\
K_{L/L^{+},p}^{\circ} &\text{if $K_p$ is parahoric}. 
\end{cases}
\end{equation*}
Therefore the assertions follow. 
\end{proof}

Combining Thorem \ref{cntr} with the results in Section \ref{warf}, we obtain the two main theorems: 

\begin{thm}\label{shc1}
\emph{Suppose that $L$ is an abelian extension of $\Q$. Then (T$_{V,K_p}$) is affirmative for any $V$ over $L$ and $K_p$ if one of the following hold: 
\begin{enumerate}
\item $L/L^{+}$ is split at all $v\mid p$,
\item the ramification index of $L^{+}/\Q$ at $p$ is an odd number,
\item $p>2$ and $[L:\Q]\not\in 32\Z$,
\item $p=2$ and $[L:\Q]\not\in 8\Z$. 
\end{enumerate}}
\end{thm}

\begin{proof}
If $L/L^{+}$ is unramified at all $v\mid p$, then the assertion follows from Theorems \ref{mtaa} and \ref{cntr}. Otherwise, the assertion is a consequence of Theorems \ref{mtas} and \ref{cntr}. 
\end{proof}

\begin{thm}\label{shc2}
\emph{
\begin{enumerate}
\item Assume $p>2$. For $d\in 32\Z$, there is an infinite family $\{L_j\}_{j\in J}$ of CM fields of degree $d$ that are abelian over $\Q$ such that (T$_{V,K_p}$) is negative for any $V$ over $L_j$ and $K_p$. Moreover, if $p\equiv 1\bmod 4$, then both the sets
\begin{equation*}
J_{\ur}:=\{j\in J\mid L_j/L_j^{+}\text{ is unramified at all }v\mid p\}
\end{equation*}
and $J\setminus J_{\ur}$ are infinite. 
\item For $d\in 8\Z$, there is an infinite family $\{L_j\}_{j\in J'}$ of CM fields of degree $d$ that are abelian over $\Q$ such that (T$_{V,K_2}$) is negative for any $V$ over $L_j$ and $K_2$. Moreover, both the sets
\begin{equation*}
J'_{\ur}:=\{j\in J'\mid L_j/L_j^{+}\text{ is unramified at all }v\mid p\}
\end{equation*}
and $J'\setminus J'_{\ur}$ are infinite. 
\end{enumerate}}
\end{thm}

\begin{proof}
By Theorem \ref{cntr} (ii), the infinite families as in Theorem \ref{ngmt} give the desired assertions. 
\end{proof}


\begin{thebibliography}{99}
\bibitem[Bor98]{Borovoi1998}
M.~Borovoi, \emph{Abelian Galois cohomology of reductive groups}, Mem.~AMS \textbf{132}, Number 626 (1998). 
\bibitem[CS77]{ColliotThelene1977}
J-L.~Colliot-Th{\'e}l{\`e}le, J-J.~Sansuc, \emph{La $R$-{\'e}quivalence sur les tores}, Ann.~sci.~de l'{\'E}.N.S.~$4^{e}$ s{\'e}rie, \textbf{10} (1977), no.~2, 175--229. 
\bibitem[CS87]{ColliotThelene1987}
J-L.~Colliot-Th{\'e}l{\`e}le, J-J.~Sansuc, \emph{Principal homogeneous spaces under flasque tori: applications}, J.~Algebra \textbf{106} (1987), no.~1, 148--205. 
\bibitem[CS07]{ColliotThelene2007}
J-L.~Colliot-Th{\'e}l{\`e}le, V.~Suresh, \emph{Quelques questions d'approximation faible pour les tores alg{\'e}briques}, Ann.~Inst.~Fourier, Grenoble \textbf{57} (2007), no.~1, 273--288. 
\bibitem[Del79]{Deligne1979}
P.~Deligne, \emph{Vari{\'e}t{\'e}s de Shimura: interpr{\'e}tation modularie, et techniques de construction de mod{\`e}les canoniques}, Automorphic forms, representations and $L$-functions (Corvallis 1977), Proc.~Sympos.~Pure Math.~XXXIII, Amer.~Math.~Soc., pp.~247--289, 1979. 
\bibitem[EM74]{Endo1974}
S.~Endo, T.~Miyata, \emph{On a classification of the function fields of algebraic tori}, Nagoya Math.~J.~\textbf{56} (1974), 85--104. 
\bibitem[HR17]{He2017}
X.~He, M.~Rapoport, \emph{Stratifications in the reduction of Shimura varieties}, manuscripta math.~\textbf{152} (2017), 317--343. 
\bibitem[HZ20]{He2020}
X.~He, R.~Zhou, \emph{On the connected components of affine Deligne-Lusztig varieties}, Duke Math.~J.~\textbf{169} (2020), no.~14, 2697--2765. 
\bibitem[Jac62]{Jacobowitz1962}
R.~Jacobowitz, \emph{Hermitian forms over local fields}, Amer.~J.~Math.~\textbf{84} (1962), 441--465. 
\bibitem[Kot92]{Kottwitz1992}
R.~E.~Kottwitz, \emph{Points on some Shimura varieties over finite fields}, J.~Amer.~Math.~Soc.~\textbf{5} (1992) no.~2, 373--444. 
\bibitem[Kot97]{Kottwitz1997}
R.~E.~Kottwitz, \emph{Isocrystals with additional structure I\hspace{-0.5mm}I}, Compos.~Math.~\textbf{109} (1997), 255--339. 
\bibitem[Mil]{Milne}
J.~S.~Milne, \emph{Introduction to Shimura varieties}, \texttt{https://www.jmilne.org/math/xnotes/svi.pdf}. 
\bibitem[PR08]{Pappas2008}
G.~Pappas, M.~Rapoport, \emph{Twisted loop groups and their affine flag varieties}, Adv.~Math.~\textbf{219} no.~1 (2008), 118--198. With an appendix by T.~Haines and M.~Rapoport. 
\bibitem[RZ96]{Rapoport1996b}
M.~Rapoport, Th.~Zink, \emph{Period spaces for $p$-divisible groups}, Annals of Mathematics Studies, vol.~141, Princeton University Press, Princeton, NJ, 1996. 
\bibitem[Rap05]{Rapoport2005}
M.~Rapoport, \emph{A guide to the reduction modulo $p$ of Shimura varieties}, Ast{\'e}risque \textbf{298} (2005), 271--318. 
\bibitem[Xia20]{Xiao2020}
L.~X.~Xiao, \emph{On the Hecke orbit conjecture for PEL type Shimura varieties}, preprint, arXiv:2006.06859, 2020. 
\bibitem[Zho20]{Zhou2020b}
R.~Zhou, \emph{Mod-$p$ isogeny classes on Shimura varieties with parahoric level structure}, Duke Math.~J.~\textbf{169} (2020), no.~15, 2937--3031. 
\end{thebibliography}
\end{document}